\documentclass[11pt]{article}

\usepackage{mathrsfs}
\usepackage{amsmath,amsthm,amssymb}
\usepackage{graphicx}
\usepackage{vector}
\usepackage{appendix}

\font\tencmmib=cmmib10 \skewchar\tencmmib '60
\newfam\cmmibfam
\textfont\cmmibfam=\tencmmib

\def\lessim{\ \lower4pt\hbox{$
		\buildrel{\displaystyle <}\over\sim$}\ }
\def\gessim{\ \lower4pt\hbox{$\buildrel{\displaystyle >}
		\over\sim$}\ }

\def\vX{\mathbf{X}}

\def\la{\langle}
\def\ra{\rangle}

\def\bY{\boldsymbol{Y}}

\newcommand{\bF}{\mathcal{F}}
\newcommand{\bC}{\mathcal{C}}
\newcommand{\bL}{\mathcal{L}}
\newcommand{\bX}{\mathbf{X}}
\newcommand{\bx}{\boldsymbol{x}}
\newcommand{\by}{\boldsymbol{y}}

\newcommand{\e}{\mathbb{E}}
\newcommand{\p}{\mathbb{P}}

\newcommand{\bW}{\boldsymbol{W}}

\newtheorem{lemma}{\bf Lemma}
\newtheorem{definition}{\bf Definition}
\newtheorem{theorem}{\bf Theorem}

\newtheorem{proposition}{\bf Proposition}

\newenvironment{Proof of lemma}{\noindent{\bf Proof of Lemma}}{\hfill$\Box$\newline}
\newenvironment{Proof of theorem}{\noindent{\bf Proof of Theorem}}{\hfill{\footnotesize${\square}$}\newline}
\newenvironment{Proof of theorems}{\noindent{\bf Proof of Theorems}}{\hfill$\Box$\newline}
\newenvironment{Proof of proposition}{\noindent{\bf Proof of Proposition}}{\hfill$\Box$\newline}
\newenvironment{Proof of propositions}{\noindent{\bf Proof of Propositions}}{\hfill$\Box$\newline}
\newenvironment{Proof of exercise}{\noindent{\it Proof of Exercise:}}{\hfill$\Box$}

\begin{document}
	
	
	\title{On the energy landscape of the mixed even $p$-spin model}
	
	\author{
		Wei-Kuo Chen\thanks{School of Mathematics, University of Minnesota. Email: wkchen@umn.edu}
		\and
		Madeline Handschy\thanks{School of Mathematics. University of Minnesota. Email: hands014@umn.edu}
		\and
		Gilad Lerman\thanks{School of Mathematics, University of Minnesota. Email: lerman@math.umn.edu}
	}
	\maketitle
	
	\begin{abstract}		
	We investigate the energy landscape of the mixed even $p$-spin model with Ising spin configurations. We show that for any given energy level between zero and the maximal energy, with overwhelming probability there exist exponentially many distinct spin configurations such that their energies stay near this energy level. Furthermore, their magnetizations and overlaps are concentrated around some fixed constants. In particular, at the level of maximal energy, we prove that the Hamiltonian exhibits exponentially many orthogonal peaks. This improves the results of Chatterjee \cite{chatt} and Ding-Eldan-Zhai \cite{DEZ}, where the former established a logarithmic size of the number of the orthogonal peaks, while the latter proved a polynomial size. Our second main result obtains disorder chaos at zero temperature and at any external field. As a byproduct, this implies that the fluctuation of the maximal energy is superconcentrated when the external field vanishes 
	and obeys a Gaussian limit law when the external field is present.
\end{abstract}



\tableofcontents
	
\section{Introduction}

Initially invented  by theoretical physicists, spin glass models are disordered spin systems created to explain the strange magnetic behavior of certain alloys. These models are typically formulated as families of highly correlated random variables, called Hamiltonians or energies in physics, indexed by configuration spaces with given metrics. This framework has also been extended to a variety of disordered models in other scientific disciplines including computer science~\cite{Bhamidi+2012, Dembo_max_cut_spin_glass_2015, Gamarnik2016,Montanari_spin_book2012} and neural networks~\cite{Amit+NN85, ChoromanskaHMAL15, LEE1986276, Peretto1986}.

Understanding the energy landscape of spin glass systems is a challenging endeavor. When the Hamiltonian of a spin glass model is defined as a Gaussian field indexed by a differentiable manifold, it is natural to study the distribution of the extrema and critical points of the Hamiltonian. Over the past two decades, physicists have intensively investigated such problems for the spherical pure $p$-spin model; see Charbonneau et al. \cite{CKPUZ}, Fyodorov-William \cite{F0}, Fydorov \cite{F1}, and Kurchan-Parisi-Virasoro \cite{KMV}. Rigorous mathematical results for the spherical pure $p$-spin and mixed $p$-spin models have also appeared in recent years. The description of complexity was presented in Auffinger-Ben Arous \cite{AB}, Auffinger-Ben Arous-Cerny \cite{ABC}, and Subag \cite{S1}, the statistics were studied in Subag \cite{S2} and Subag-Zeitouni \cite{SZ}, and the formula for the maximal energy was obtained in Chen-Sen \cite{ArnabChen15} and Jagannath-Tobasco \cite{JT}. 
 
Studies related to energy statistics of spin glass models with discrete configuration spaces appear in Bovier-Kurkova \cite{BKU041,BKU042,BK05,BK06,BKU06,BK07}, Bovier-Klimovsky \cite{BKL08}, and Ben Arous-Gayrard-Kuptsov~\cite{GGV}. These works consider a wide class of examples including the random energy model, the generalized random energy model, the Sherrington-Kirkpatrick (SK) model, the pure $p$-spin model and the random partition problems. More recently, Chatterjee \cite{chatt} revolutionized the study of energy landscapes of spin glass models by establishing the so-called {\it superconcentration}, {\it disorder chaos}, and {\it multiple peaks} via an interpolation method for general discrete Gaussian fields. In the context of the SK model, his results imply that near the maximal energy the system exhibits multiple peaks in the sense that there exist logarithmically many spin configurations around the maximal energy and they are nearly orthogonal to each other. This picture extends to any energy level. Another paper of Chatterjee \cite{chatt1} establishes relationships between superconcentration, disorder chaos, and multiple peaks for more general Gaussian fields than spin glass models. Chatterjee's general results \cite{chatt1,chatt,chattbook} on multiple peaks were improved by Ding-Eldan-Zhai \cite{DEZ} and as a byproduct, they showed that the number of orthogonal peaks in the SK model is at least of polynomial order. 

Theoretical physicists believe that the multiple peaks discussed above occupy a large portion of the configuration space (see, e.g., M\'{e}zard-Parisi-Virasoro \cite{MPV}). The main goal of this paper is to establish a description of this belief for the SK model as well as the mixed $p$-spin model. We show that for any given energy level between zero and the maximal energy, one can construct exponentially many spin configurations with energies around the given level. Furthermore, they possess the same magnetization and are equidistant from each other. In particular, at the level of maximal energy, our result strengthens the aforementioned results in \cite{chattbook,DEZ} by showing that the number of orthogonal peaks is of exponential order.

\subsection{The mixed even $p$-spin model}
We first review the mixed even $p$-spin model with Ising spin configurations. For any integer $N\geq 1$, let $\Sigma_N:=\{-1,+1\}^N$ be the set of spin configurations. The Hamiltonian of the mixed even $p$-spin model is defined as
$$
H_N(\sigma)=\sum_{p\in 2\mathbb{N}}c_pH_{N,p}(\sigma)
$$
for $\sigma\in\Sigma_N,$ where the sum is over all even $p$ and $H_{N,p}$ is the pure $p$-spin Hamiltonian defined by
$$
H_{N,p}(\sigma)=\frac{1}{N^{(p-1)/2}}\sum_{1\leq i_1,\ldots,i_p\leq N}g_{i_1,\ldots,i_p}\sigma_{i_1}\cdots\sigma_{i_p}.
$$
Here, for all $1\leq i_1,\ldots,i_p\leq N$ and $p\in 2\mathbb{N},$ $g_{i_1,\ldots,i_p}$'s are i.i.d. standard Gaussian random variables  We assume that the sequence $(c_p)_{p\in 2\mathbb{N}}$ satisfies $c_p\neq 0$ for at least one $p$ and it decays fast enough, for instance, $\sum_{p\in 2\mathbb{N}}2^pc_p^2<\infty$, so that the infinite sum $H_N$ converges a.s.
Under these assumptions, one readily computes that
$$
\e H_{N}(\sigma^1)H_N(\sigma^2)=N\xi(R(\sigma^1,\sigma^2)),
$$
where
\begin{align}
\label{xi}
\xi(s):=\sum_{p\in 2\mathbb{N}}c_p^2s^p
\end{align}
and $$R(\sigma^1,\sigma^2):=\frac{1}{N}\sum_{i=1}^N\sigma_i^1\sigma_i^2$$ is the overlap between $\sigma^1$ and $\sigma^2$. The classical SK model corresponds to $\xi(s)=s^2/2$. If one replaces the space $\Sigma_N$ by the sphere $\{\sigma\in\mathbb{R}^N:\sum_{i=1}^N\sigma_i^2=N\}$, then the above Hamiltonian $H_N$ is called the spherical mixed even $p$-spin model. 

The mixed even $p$-spin model with external field is defined by
$$
H_N^h(\sigma)=H_N(\sigma)+hNm_N(\sigma),
$$
where $h\geq 0$ denotes the strength of the external field and $$
m_N(\sigma):=\frac{1}{N}\sum_{i=1}^N\sigma_i
$$
is called the magnetization of $\sigma.$ Denote by $L_N^h$ the maximal energy of $H_N^h$, that is,
$$
L_N^h:=\max_{\sigma\in\Sigma_N}\frac{H_N^h(\sigma)}{N}.
$$
Here, unlike the tradition in physics, we consider the maximum rather than the minimum of the energy. Nonetheless, they differ essentially  only by a negative sign.
It is well-known (see, for instance, \cite{Pan}) that for all $h\in \mathbb{R},$
\begin{align}
\label{add:eq2}
M(h):=\lim_{N\rightarrow\infty}L_N^h
\end{align}
exists and is equal to a nonrandom number a.s., which can be computed through the Parisi formula for the free energy at positive temperature. Recently, Auffinger-Chen \cite{AC15} proved that $M(h)$ can also be expressed as a Parisi-type formula for the limiting maximal energy $L_N^h$ (see \eqref{parisi} below).

\subsection{Energy landscape}

We study the energy landscape of the Hamiltonian $H_N$. In order to investigate the behavior of the spin configurations at different energy levels, we introduce an auxiliary function,
\begin{align}\label{eq-5}
E(h):=M(h)-hM'(h),\,\,\forall h\geq 0,
\end{align}
where the differentiability of $M$ is guaranteed by Proposition \ref{sec5:prop1} in Subsection \ref{sec:prop-0}.
 The following proposition states that by varying $h$ between zero and infinity, the function $E(h)$ continuously scans all energy levels of $H_N$  between zero and its maximal energy.

\begin{proposition}\label{prop-0}
$E(h)$ is  nonincreasing and continuous with $E(0)=M(0)$ and $\lim_{h\rightarrow\infty}E(h)=0$.
\end{proposition}

Our main result on the energy landscape of the Hamiltonian $H_N$ is formulated in Theorem \ref{thm1} below. It uses the notion of the overlap constant, $q_h$.  We postpone its definition to Subsection \ref{sec4.2}, but remark that $q_h$ depends only on $h$ and $\xi$ and satisfies $0<q_h<1$ if $h\neq 0$ and $q_h=0$ if $h=0$. 
\begin{theorem}\label{thm1}
     Let $h\geq 0.$ For any $\varepsilon>0,$ there exists a constant $K>0$ depending on $h,\xi,\varepsilon$ such that for any $N\geq 1,$ with probability at least $1-Ke^{-N/K}$, there exists $S_N(h)\subset\Sigma_N$ such that the following hold:
	\begin{itemize}
		\item[$(i)$] $|S_N(h)|\geq e^{N/K}$,
		\item[$(ii)$] For any $\sigma\in S_N(h)$,
		$
		\Bigl|\frac{H_N(\sigma)}{N}-E(h)\Bigr|< \varepsilon.
		$
		\item[$(iii)$] For any $\sigma\in S_N(h)$, $|m_N(\sigma)-M'(h)|<\varepsilon.$
		\item[$(iv)$] For any distinct $\sigma,\sigma'\in S_N(h)$, $|R(\sigma,\sigma')-q_h |< \varepsilon$.
	\end{itemize}
\end{theorem}

Theorem \ref{thm1} means that sufficiently near the energy level $E(h)$ of $H_N/N$, there exist exponentially many spin configurations that are approximately equidistant from each other and their magnetizations are near $M'(h)$. Both $q_h=0$ and $M'(h)=0$ when $h=0.$ Consequently, Theorem~\ref{thm1} assures the existence of exponentially many orthogonal peaks with zero magnetization.

As we mentioned before, in the SK model the energy landscape at the maximal energy was studied previously by Chatterjee \cite{chatt}. He showed that with probability at least $1-K(\log N)^{-1/12}$ for some constant $K$, there exists a set of spin configurations of size at least $(\log N)^{1/8}$ that satisfy $(ii)$ and $(iv)$ with $q_h=0$. In Ding-Eldan-Zhai \cite{DEZ}, the size of this set was improved to be of polynomial order. Both works specified the dependence of $\varepsilon$ on $N$.  Theorem \ref{thm1} here improves these results (without quantitative dependence of $\varepsilon$ on $N$) by extending the number of orthogonal peaks to be of exponential order and  establishing analogous geometric structure at all energy levels. 

\subsection{Chaos in disorder}

We formulate our second main result: disorder chaos at zero temperature for the Hamiltonian $H_N^h$. This is interesting on its own, but is also crucial to the construction of the set $S_N(h)$ of Theorem~\ref{thm1}. Generally, chaos in spin glasses studies the instability of the system subject to a small perturbation of certain external parameters such as temperature, disorder, or external field. It is a very old subject that has received a lot of attention in the physics community; see Rizzo \cite{Rizzo} for an up-to-date survey. In recent years, intensive mathematical progress on chaos in disorder has been made for the mixed even $p$-spin models: At positive temperature, Chatterjee \cite{chatt} and Chen \cite{C12,C14} considered the Ising spin case, while Chen-Hsieh-Hwang-Sheu~\cite{C15} studied the spherical spin case. More recently, disorder chaos at zero temperature was obtained in the spherical model by Chen-Sen \cite{ArnabChen15}.

Our investigation aims to establish chaos in disorder for the Hamiltonian $H_N^h$ at zero temperature. More precisely, we are interested in the behavior of the maximizer of $L_N^h$ when a perturbation is applied to the disorder. Let $t\in[0,1]$ be a coupling parameter. Consider two Hamiltonians
\begin{align*}
H_{N,t}^{1,h}(\sigma^1 )&:=\sqrt{t}H_{N}(\sigma^1 )+\sqrt{1-t}H_{N}^1(\sigma^1 )+h\sum_{i=1}^N\sigma^1 _i,\\
H_{N,t}^{2,h}(\sigma^2  )&:=\sqrt{t}H_N(\sigma^2  )+\sqrt{1-t}H_N^2(\sigma^2  )+h\sum_{i=1}^N\sigma^2  _i,
\end{align*}
where $H_N^1$ and $H_N^2$ are two i.i.d. copies of $H_N$. One immediately sees that $H_{N,t}^1=H_{N,t}^2$ if $t=1$, while for $0<t<1,$ there is a certain amount of decoupling between the two systems controlled by the parameter $t$. 

We measure chaos in disorder by the cross overlap between the maximizers of $H_{N,t}^{1,h}$ and $H_{N,t}^{2,h}.$ By cross overlap, we mean the overlap between spin configurations from different Hamiltonians. Let $\sigma_{t,h}^{1}$ and $\sigma_{t,h}^2$ be any maximizers of $H_{N,t}^{1,h}$ and $H_{N,t}^{2,h}$, respectively. Note that the Hamiltonians $H_{N}, H_N^1,H_N^2$ are summations of even $p$-spin interactions. If $h\neq 0,$ then these maximizers are almost surely unique. Nonetheless,  if $h=0,$ then the two Hamiltonians $H_{N,t}^{1,0}$ and $H_{N,t}^{2,0}$ are symmetric, i.e., $H_{N,t}^{j,0}(\sigma)=H_{N,t}^{j,0}(-\sigma)$ for $j=1,2$ and yield
\begin{align*}
\mbox{Argmax}_{\sigma\in\Sigma_N}H_{N,t}^{1,0}(\sigma)&=\{\sigma_{t,0}^1,-\sigma_{t,0}^1\},\\
\mbox{Argmax}_{\sigma\in\Sigma_N}H_{N,t}^{2,0}(\sigma)&=\{\sigma_{t,0}^2,-\sigma_{t,0}^2\},
\end{align*}
from which one sees that the absolute value of the cross overlap between any maximizers of $H_{N,t}^{1,h}$ and $H_{N,t}^{2,h}$ is always equal to
\begin{align*}
|R(\sigma_{t,0}^1,\sigma_{t,0}^2)|.
\end{align*} 
Based on these, we measure the instability of the system with respect to disorder perturbation by $R(\sigma_{t,h}^1,\sigma_{t,h}^2)$ when $h\neq 0$ and by $|R(\sigma_{t,0}^1,\sigma_{t,0}^2)|$ when $h=0.$ It is clear that these two quantities are equal to one if $t=1.$ However, if $0<t<1$, Theorem \ref{thm-1} below shows that they are concentrated around a constant strictly less than one no matter how close $t$ is to $1$. 

\begin{theorem}\label{thm-1}
	Let $t\in (0,1)$ and $\varepsilon>0.$ The following two statements hold for any $\xi$ defined in \eqref{xi}: 
	\begin{itemize}
		\item[$(i)$] If $h=0,$ there exists a constant $K>0$ depending on $t,\xi,\varepsilon$ such that \begin{align*}
		\p\bigl(|R(\sigma_{t,0}^1  ,\sigma_{t,0}^2 )|\geq \varepsilon\bigr)\leq K\exp\Bigl(-\frac{N}{K}\Bigr),\,\,\forall N\geq 1.
		\end{align*}
		\item[$(ii)$] If $h\neq 0,$ there exists a constant $K>0$ depending on $t,h,\xi,\varepsilon$ such that
		\begin{align*}
		\p\bigl(|R(\sigma_{t,h}^1  ,\sigma_{t,h}^2 )-q_{t,h}|\geq \varepsilon\bigr)\leq K\exp\Bigl(-\frac{N}{K}\Bigr),\,\,\forall N\geq 1,
		\end{align*}
       where the constant $q_{t,h}$ relies only on $\xi,h,t$ and satisfies $q_{t,h}\in(0,q_h)$. Here, $q_h$ is the same constant as the one in Theorem \ref{thm1}.
	\end{itemize}

\end{theorem}

Theorem \ref{thm-1}$(i)$ means that as long as there is no external field, the maximizers of $H_{N,t}^{1,0}$ are orthogonal to the maximizers of $H_{N,t}^{2,0}$. In the case that $h\neq 0$, Theorem \ref{thm-1}$(ii)$ states that $\sigma_{t,h}^1$ and $\sigma_{t,h}^2$ preserve a strictly positive distance $\sqrt{2(1-q_{t,h})}$  as $N$ tends to infinity. The key point is that even when $t$ approaches $1$, this distance remains at least $\sqrt{2(1-q_h)}>0.$ This indicates instability of the maximizers due to the perturbation of the disorder. Theorem \ref{thm-1} was previously obtained in the context of the spherical mixed even $p$-spin model by Chen-Sen \cite{ArnabChen15}. However, as explained in Subsection \ref{add:sub1}, the proof of Theorem \ref{thm-1} is fundamentally different than that of the spherical case. 

Finally, we derive two consequences of Theorem \ref{thm-1} on the
fluctuation properties of the maximal energy:

\begin{theorem}\label{thm5} For any $\xi$ defined in \eqref{xi}, the following two limits hold.
\begin{itemize}
	\item[$(i)$] If $h=0,$ then
	\begin{align*}
	\lim_{N\rightarrow\infty}N\mbox{Var}(L_N^h)=0.
	\end{align*}
	\item[$(ii)$] If $h>0,$ then
	\begin{align*}
	\lim_{N\rightarrow\infty}d_{TV}\bigl(\sqrt{N}\bigl(L_N^h-\e L_N^h\bigr),g\bigr)=0,
	\end{align*}
	where $d_{TV}$ is the total variation distance and $g$ is a centered Gaussian random variable with variance $
	\int_0^1\xi(q_{t,h})dt.
	$
\end{itemize}

\end{theorem}

	The Poincar\'{e} inequality implies that for all $h\geq 0$, there exists some $C>0$ such that $\mbox{Var}(L_N^h)\leq C/N$ for any $N\geq 1$. Theorem~\ref{thm5}$(i)$ states that if the external field vanishes, $L_N^h$ is superconcentrated in the sense that it has a faster self-averaging rate than the Poincar\'{e} bound. We emphasize that this result applies to the SK model and it thus solves Open Problem~3.2 in Chatterjee \cite{chattbook}. We refer the readers to Chatterjee~\cite{chatt1} for some earlier results on superconcentration, where quantitative bounds were obtained for certain choices of $\xi$, which exclude the SK model. While $L_N^h$ exhibits superconcentration in the absence of external field, Theorem \ref{thm5}$(ii)$ shows that $L_N^h$ has a Gaussian fluctuation if the external field is present. In other words, the Poincar\'{e} inequality is sharp whenever $h>0$. An analogue of Theorem \ref{thm5} was established by Chen-Sen \cite{ArnabChen15} for the spherical mixed $p$-spin model. Its derivation relies on an analogue of Theorem \ref{thm-1} for the spherical case and an application of Stein's method for normal approximation. The  proof of Theorem \ref{thm5} is practically identical to that of the spherical case in  Chen-Sen~\cite{ArnabChen15}.

\subsection{Main ideas of proofs}\label{add:sub1}

Our proof of Theorem \ref{thm1} is motivated by Chatterjee \cite{chatt}. To construct multiple peaks, he first established chaos in disorder for the SK model at positive temperature and zero external field by deriving upper bounds for the moments of the cross overlap. These bounds depend on the temperature and the coupling parameter $t$ and they asymptotically vanish as $N$ tends to infinity  when $t<1$. He then used these bounds with an adjusted temperature to select logarithmically many orthogonal peaks.  The same bounds were later employed in Ding-Eldan-Zhai \cite{DEZ} to obtain a polynomially many peaks. The construction of peaks in these works requires a careful adjustment of both the coupling parameter and the inverse temperature. Such a simultaneous adjustment is restrictive and makes it very difficult to derive an exponential number of orthogonal peaks.  

To overcome this difficulty, we verify disorder chaos at zero temperature with exponential tail control in Theorem \ref{thm-1}, instead of considering the moments of the cross overlap at positive temperature. We explain the basic idea of establishing this in the case where $h=0$ (the case of $h\neq 0$ is similar). We show that for any $0<t,\varepsilon<1$, there exist constants $\eta,K>0$ depending only on $\xi, t,\varepsilon$ such that with probability at least $1-Ke^{-N/K}$,
\begin{align}
\begin{split}\label{add:eq9}
&\frac{1}{N}\max_{(\sigma^1,\sigma^2)\in\Sigma_N^2:|R(\sigma^1,\sigma^2)|\geq\varepsilon}\bigl(H_{N,t}^{1,0}(\sigma^1)+H_{N,t}^{2,0}(\sigma^2)\bigr)\\
&\leq \frac{1}{N}\max_{\sigma^1\in\Sigma_N}H_{N,t}^{1,0}(\sigma^1)+\frac{1}{N}\max_{\sigma^2\in\Sigma_N}H_{N,t}^{2,0}(\sigma^2)-\eta
\end{split}
\end{align}
for all $N\geq 1$. If this is valid, the probability of $|R(\sigma_{t,0}^{1},\sigma_{t,0}^2)|\geq \varepsilon$ must be bounded by $Ke^{-N/K}$ due to the positivity of $\eta$ and the optimality of $\sigma_{t,0}^1,\sigma_{t,0}^2$, which yields Theorem~\ref{thm-1}$(i)$. Generally, proving \eqref{add:eq9} is a very challenging task for arbitrary Gaussian fields. To accomplish this in the current setting we rely on the fact that the two sides of \eqref{add:eq9} can be controlled by the Parisi formula and the Guerra-Talagrand bound at zero temperature. The analysis of these formula and bound is based on an analogous study of the mixed $p$-spin model at positive temperature in the framework of Chen \cite{C14} and the recent development of the Parisi formula at zero temperature by Auffinger-Chen~\cite{AC16}.  Once chaos in disorder at zero temperature is established, the construction of nearly orthogonal peaks in Theorem \ref{thm1} is similar to Chatterjee~\cite{chatt}. Taking into account the external field allows us to select exponentially many equidistant spin configurations at all energy levels of $H_N.$ 

As mentioned before, a version of Theorem \ref{thm-1} for the spherical mixed $p$-spin model appeared in Chen-Sen \cite{ArnabChen15}. The approaches in both the Ising and spherical cases require a series of sophisticated techniques and results about the properties of Parisi's formulas at positive and zero temperatures as well as Guerra-Talagrand bounds. In the spherical case, these formulas and bounds admit explicit and simple expressions. They yield great simplifications in  controlling the coupled maximal energy of $H_{N,t}^{1,h}$ and $H_{N,t}^{2,h}$. However, in the Ising case, the Parisi formulas and Guerra-Talagrand bounds are formulated in a more delicate way through a semi-linear parabolic differential equation and its two-dimensional extension, called the Parisi PDEs. The analysis in the present paper is generally more subtle than that in the spherical case.  

\subsection{Open questions}

In view of Theorem \ref{thm1}, there are two closely related open questions of great interest. The first is to establish a version of Theorem \ref{thm1} with an error term $\varepsilon$ depending on $N.$ 

The second problem is to further explore the structure of the energy landscape of $H_N$. From the methodology presented in the rest of the paper, it seems possible that one can construct another set $S_N'(h)\subset\Sigma_N$ at the same energy level as that of $S_N(h)$ such that the spin configurations within $S_N'(h)$ are again equidistant to each other and also preserve a constant distance to all elements in $S_N(h).$ Besides, it also seems possible that for any $0\leq h<h',$ one can construct $S_N(h)$ and $S_N(h')$ at different energy levels such that Theorem \ref{thm1} holds for both energy levels and the distance between $S_N(h)$ and $S_N(h')$ is a fixed constant. These can be verified if one can establish results for chaos in external field $h$ and chaos in mixture parameters $(c_p)_{p\geq 2},$ similar to chaos in disorder in Theorem \ref{thm-1}.

\subsection{Organization of the paper}
The paper is organized as follows. Section~\ref{sec:Parisi_formula} focuses on the Parisi formula for the maximal energy $L_N^h$ and establishes the uniqueness of the Parisi measure. Furthermore, Section~\ref{sec:Parisi_formula} proves the non-triviality of the Parisi measure and also derives some consistency equations for this measure.  Section~\ref{sec:disorder_chaos} presents the proof for Theorem~\ref{thm-1} based on the results in Section~\ref{sec:Parisi_formula}. Section \ref{Sec5} verifies Proposition \ref{prop-0} and Theorem~\ref{thm1}. The proof of Theorem \ref{thm5} is omitted as it is identical to those of Theorems~$4$ and $5$ in \cite{ArnabChen15}.
Some technical results regarding the regularity of the Parisi PDEs  at zero temperature are left to the appendix. The paper frequently uses ideas of \cite{C14} and skips some proofs of results whenever they clearly follow an identical argument of \cite{C14}.

\vskip 0.3cm

	{\noindent \bf Acknowledgements.} 
	The authors are indebted to N. Krylov and M. Safonov for illuminating discussions on the regularity properties of the Parisi PDE. They thank S. Chatterjee, D. Panchenko, and the anonymous referees for a number of suggestions regarding the presentation of the paper. The research of W.-K. C. is partially supported by NSF grant DMS-16-42207 and Hong Kong research grants council GRF-14-302515. The research of M. H. and G. L. is partially supported by NSF grant DMS-14-18386.

\section{Parisi formula}
\label{sec:Parisi_formula}

The Parisi formula for the free energy in the mixed even $p$-spin model was first verified in the celebrated work of Talagrand \cite{Tal03}. Later Panchenko \cite{Pan00} validated it for more general mixtures including odd $p$-spin interactions. More recently, Auffinger-Chen \cite{AC16} extended Parisi's formula to the maximal energy of $H_N^h.$ Their formulation is described below.
Let $\mathcal{U}$ be the collection of all functions $\gamma$ on $[0,1)$ induced by some measure $\mu$, i.e., $\gamma(s)=\mu([0,s])$ and satisfying
$$
\int_0^1\gamma(s)ds<\infty.
$$
We equip $\mathcal{U}$ with the $L^1$ distance with respect to the Lebesgue measure on $[0,1)$.
For each $\gamma\in \mathcal{U}$, let $\Phi_\gamma$ be the weak solution of
\begin{align}
\label{pde}
\partial_s\Phi_\gamma(s,x)&=-\frac{\xi''(s)}{2}\bigl(\partial_{xx}\Phi_\gamma(s,x)+\gamma(s)\bigl(\partial_x\Phi_\gamma(s,x)\bigr)^2\bigr)
\end{align}
for $(s,x)\in [0,1)\times\mathbb{R}$ with boundary condition $\Phi_\gamma(1,x)=|x|.$ Here the existence of $\Phi_\gamma$ is assured by Proposition \ref{pde:prop1} below. We call $\Phi_\gamma$ the Parisi PDE solution. Define
$$
\mathcal{P}_h(\gamma)=\Phi_\gamma(0,h)-\frac{1}{2}\int_0^1\xi''(s)s\gamma(s)ds,\,\,\gamma\in\mathcal{U}.
$$
The Parisi formula \cite[Theorem 1]{AC16} asserts that the limiting maximal energy of $H_N^h$ can be written as a variational problem,
\begin{align}\label{parisi}
M(h)=\lim_{N\rightarrow\infty}L_N^h=\inf_{\gamma\in \mathcal{U}}\mathcal{P}_h(\gamma), a.s.
\end{align}
Note that from \cite{AC16}, $\mathcal{P}_h$ is a continuous functional and the minimizer of the variational problem exists. As one will see in Subsection \ref{unique}, this minimizer is indeed unique. It will be called the Parisi measure and denoted by $\gamma_{h} $ throughout the remainder of the paper.


\subsection{Uniqueness}\label{unique}

Uniqueness of the minimizer of the Parisi formula for the free energy of the mixed $p$-spin model was proved in Auffinger-Chen \cite{AC14}. Our main result here is an extension of \cite{AC14} at zero temperature.

\begin{theorem}\label{thm4}
	The Parisi formula \eqref{parisi} has a unique minimizer.
\end{theorem}

The proof of Theorem \ref{thm4} will be based on \cite{AC14}. The aim is to show that the functional $\gamma\mapsto \Phi_\gamma$ is strictly convex on the space $\mathcal{U}$ via the stochastic optimal control representation for $\Phi_\gamma$ (see Theorem \ref{thm0} below). While the formula considered in \cite{AC14} restricts to those $\gamma$ with $\gamma(1-)\leq 1$ and the boundary condition of the Parisi PDE has nice regularities, the added difficulties in the current situation are that it is possible that $\gamma(1-)=\infty$ and the boundary condition has a nondifferentiable point at $0$. The following proposition shows that we still have good regularity properties on the spatial derivatives of $\Phi_\gamma$ as long as the time variable is away from $1$. This is enough for us to prove Theorem \ref{thm4}. Denote by $\mathcal{U}_d\subset\mathcal{U}$ the set of all $\gamma$'s induced by atomic measures and by $\mathcal{U}_c$ the set of all $\gamma$'s that are induced by measures without atoms. Note that $\gamma\in \mathcal{U}_c$ means that $\gamma$ is continuous on $[0,1).$ 

\begin{proposition}\label{pde:prop1}
	Let $\gamma\in\mathcal{U}$. The following statements hold:
	\begin{itemize}
		\item[$(i)$] The weak solution $\Phi_\gamma$ exists and is unique. Moreover, it satisfies
		\begin{align}\label{lip}
		|\Phi_\gamma(s,x)-\Phi_{\gamma'}(s,x)|\leq \frac{3}{2}\int_0^1\xi''(s)|\gamma(s)-\gamma'(s)|ds
		\end{align}
		for any $(s,x)\in[0,1]\times\mathbb{R}$ and $\gamma,\gamma'\in\mathcal{U}.$
		\item[$(ii)$] For any $k\geq 1$, $\partial_x^k\Phi_\gamma$ exists and is continuous on $[0,1)\times\mathbb{R}$. Furthermore, for any $s_1\in(0,1)$,
		\begin{align*}
		\sup_{(s,x)\in[0,s_1]\times\mathbb{R}}|\partial_x^k\Phi_{\gamma}(s,x)|<F_k(\gamma(s_1)),
		\end{align*}
		where $F_k$ is a continuous function defined on $[0,\infty)$ independent of $\gamma.$
		\item[$(iii)$] In either $\mathcal{U}_d$ or $\mathcal{U}_c$, there exists $(\gamma_n)_{n\geq 1}$ with weak limit $\gamma$ such that for any $k\geq 0,$
		\begin{align*}
		\lim_{n\rightarrow\infty}\sup_{(s,x)\in[0,s_1)\times[-M,M]}\bigl|\partial_x^k\Phi_{\gamma_n}(s,x)-\partial_x^k\Phi_\gamma(s,x)\bigr|=0
		\end{align*}	
		for any $s_1\in(0,1)$ and $M>0.$
	\end{itemize}
\end{proposition}

We defer the proof of Proposition \ref{pde:prop1} to the appendix. In what follows, we recall the stochastic optimal control for $\Phi_\gamma$ from \cite{AC16}. Let $W=(W(w))_{0\leq w\leq 1}$ be a standard Brownian motion. For $0\leq r<s\leq 1$, denote by $D[r,s]$ the collection of all progressively measurable processes $u$ on $[r,s]$ with respect to the filtration generated by $W$ and satisfying $\sup_{0\leq s\leq 1}|u(s)|\leq 1.$ We equip the space $D[r,s]$ with the metric
\begin{align}
\label{metric}
d_0(u,u'):=\Bigl(\e \int_{r}^s\e |u(w)-u'(w)|^2dw\Bigr)^{1/2}.
\end{align}
 Let $\gamma\in\mathcal{U}.$ For any $x\in\mathbb{R}$ and $u\in D[r,s],$ define
\begin{align*}
F^{r,s}(u,x)=\e\left[C^{r,s}(u,x)-L^{r,s}(u)\right],
\end{align*}
where
\begin{align*}
\label{thm0:eq2}
C^{r,s}(u,x)&=\Phi_\gamma\left(t,x+\int_r^s\gamma(w)\xi''(w)u(w)dw+\int_r^s\xi''(w)^{1/2}dW(w)\right),\\
L^{r,s}(u)&=\frac{1}{2}\int_r^s\gamma(w)\xi''(w)u(w)^2dw.
\end{align*}

\begin{theorem}\label{thm0} Let $0\leq r\leq s\leq 1.$ For any $\gamma\in\mathcal{U}$,
	\begin{equation}\label{MaxFormula}
	\Phi_\gamma(r,x)=\max\left\{F^{r,s}(u,x)|u\in D[r,s]\right\},
	\end{equation}
	where the maximum in \eqref{MaxFormula} is attained by
	\begin{align}\label{max}
	u_\gamma(w)&=\partial_x\Phi_\gamma(w,X_\gamma(w)).
	\end{align}
	Here $(X_\gamma(w))_{r\leq w\leq s}$ is the strong solution to
	\begin{align*}
	dX_\gamma(w)&=\gamma(w)\xi''(w)\partial_x\Phi_\gamma(w,X_\gamma(w))dw+\xi''(w)^{1/2}dW(w),\\
	X_\gamma(r)&=x.
	\end{align*}	
\end{theorem}

Theorem \ref{thm0} is essentially taken from \cite{AC16}, where it was shown to be valid for $\gamma\in \mathcal{U}_c$ by a direct application of It\^{o}'s formula. For arbitrary $\gamma\in \mathcal{U}$, it remains true by using Proposition \ref{pde:prop1} combined with an approximation argument similar to the proof of \cite[Theorem 3]{AC14}. We omit the details here. Next, we show that $\Phi_\gamma$ is convex in $\gamma,$ which will play an essential role in proving the strict convexity of $\gamma\mapsto\Phi_\gamma.$

\begin{lemma}
	\label{lem4}
	For $\gamma  _0,\gamma  _1\in\mathcal{U}$, $x_0,x_1\in\mathbb{R}$ and $\theta\in[0,1],$ denote
	\begin{align}
	\begin{split}\label{eq-2}
	\gamma  _\theta&=(1-\theta)\gamma  _0+\theta\gamma  _1,\\
	x_\theta&=(1-\theta)x_0+\theta x_1.
	\end{split}
	\end{align}
	Then
	\begin{align}
	\label{thm1:eq1}
	\Phi_{\gamma_\theta}(r,x_\theta)&\leq (1-\theta)\Phi_{\gamma_0}(r,x_0)+\theta\Phi_{\gamma_1}(r,x_1)
	\end{align}
	for any $r\in[0,1]$.
\end{lemma}

\begin{proof}
	Let $\gamma  _0,\gamma  _1\in\mathcal{U}$, $x_0,x_1\in\mathbb{R}$, $\theta\in[0,1]$, $0\leq r\leq s\leq 1$ and $u\in D[r,s].$ Let $a = 0, \theta,$ or $1$. Denote by
	\begin{align*}
	F_a^{r,s},C_a^{r,s},L_a^{r,s},
	\end{align*}
	the functionals defined in the variational formulas corresponding respectively to $\gamma  _a$:
	\begin{align}
	\begin{split}
	\label{eq16}
	\Phi_{\gamma_a}(s,x_a)&=\max\{F_\theta^{r,s}(u,x_a)|D[r,s]\}.
	\end{split}
	\end{align}
	Let $0\leq r\leq 1$ and take $s=1$. Suppose that $u\in D[r,1].$ Observe that
	\begin{align*}
	L_\theta^{r,1}(u)&=(1-\theta) L_0^{r,1}(u)+\theta L_1^{r,1}(u)
	\end{align*}
	and from the convexity of $|x|$,
	\begin{align*}
	C_\theta^{r,1}(u,x_\theta)&\leq (1-\theta) C_0^{r,1}(u,x_0)+\theta C_1^{r,1}(u,x_1).
	\end{align*}
	These together imply
	\begin{align*}
	F_\theta^{r,1}(u,x_\theta)&\leq (1-\theta) F_0^{r,1}(u,x_0)+\theta F_1^{r,1}(u,x_1).
	\end{align*}
	Since this is true for any $u\in D[r,1],$ the representation formula \eqref{eq16} gives
	\eqref{thm1:eq1}.
\end{proof}

In order to prove strict convexity, we need two results regarding the uniqueness as well as some properties of the optimizer in \eqref{MaxFormula}. The first result gives the uniqueness of the maximizer $u_\gamma$ if $\gamma(r)>0.$

\begin{lemma}[Uniqueness]
	\label{prop1}
	Let $\gamma\in\mathcal{U}$ and $0\leq r<s\leq 1$ with $\gamma(r)>0.$ If $u$ attains the maximal value of the variational representation \eqref{MaxFormula}, then $u=u_\gamma$.
\end{lemma}

Lemma \ref{prop1} is proved using Proposition~\ref{pde:prop1}$(ii)$ in an argument identical to that of \cite[Lemma 5]{C14}.  We will not reproduce the details here.
Lemma \ref{lem} below gives the usual results that $\partial_x\Phi_\gamma(w,X_\gamma(w))$ is a martingale and $\partial_{xx}\Phi_{\gamma}(w,X_\gamma(w))$ is a semi-martingale.

\begin{lemma}\label{lem}
	Let $\gamma  \in\mathcal{U}$, $0\leq r\leq s<1$, and $x\in\mathbb{R}$. For any $r\leq a\leq b\leq s$, we have
	\begin{align}\label{eq10}
	\partial_x\Phi_\gamma (b,X_\gamma(b))-\partial_x\Phi_\gamma (a,X_\gamma(a))&=\int_a^b\xi''  (w)^{1/2}\partial_{xx}\Phi_\gamma (w,X_\gamma(w))dW(w)
	\end{align}
	and
	\begin{align}
	\begin{split}\label{eq11}
	&\partial_{xx}\Phi_\gamma (b,X_\gamma(b))-\partial_{xx}\Phi_\gamma (a,X_\gamma(a))\\
	&=-\int_a^b\gamma  (w)\xi''  (w)(\partial_{xx}\Phi_\gamma (w,X_\gamma(w)))^2dr+\int_a^b\xi''  (w)^{1/2}\partial_{x}^3\Phi_\gamma (w,X_\gamma(w))dW(w).
	\end{split}
	\end{align}
\end{lemma}

If $\gamma\in \mathcal{U}_c$, Lemma \ref{lem} is easily verified by a standard application of It\^{o}'s formula as $\partial_t\partial_{xx}\Phi_\gamma$ is continuous on $[0,1)\times\mathbb{R}$ in this case. For the general case, one can argue by an approximation procedure via  Proposition \ref{pde:prop1}$(iii)$, similar to that of \cite[Proposition 3]{AC14} with some minor modifications. Again we will omit these details. Next, we state a crucial property of $\partial_x\Phi_\gamma.$

\begin{lemma}\label{lem3}
For any $\gamma\in\mathcal{U}$ and $s\in[0,1)$, $\partial_x\Phi_\gamma(s,\cdot)$ is odd and strictly increasing.
\end{lemma}

\begin{proof}
	Since the boundary condition $|x|$ is even, it is easy to see that $\Phi_\gamma(s,\cdot)$ is also even and thus $\partial_x\Phi_\gamma(s,\cdot)$ is odd. From Lemma \ref{lem4}, we know that $\Phi_\gamma(s,\cdot)$ is convex. This implies that $\partial_x\Phi_\gamma(s,\cdot)$ is nondecreasing. To see that $\partial_x\Phi_{\gamma}(s,\cdot)$ is strictly increasing, for any two distinct $x_0,x_1\in\mathbb{R}$ and $\theta\in(0,1)$, let $u_{\gamma}^{\theta}$ be the maximizer and $X_{\gamma}^{\theta}$ be the SDE solution in the representation \eqref{MaxFormula} for $\Phi_\gamma(s,x_\theta)$. Note that from Girsanov's theorem \cite[Theorem 5.1]{KS}, the distribution of
	$$
	\Delta:=\int_s^1\xi''\gamma u_{\gamma}^\theta dw+\int_s^1{\xi''}^{1/2}dW
	$$
	is Gaussian under some change of measure. Therefore, $\Delta$ is supported on $\mathbb{R}$, so
	\begin{align*}
	\p\bigl((x_0+\Delta)(x_1+\Delta)<0\bigr)>0.
	\end{align*}
	Thus with positive probability,
	\begin{align*}
	|X_{\gamma}^\theta|&=|(1-\theta)\bigl(x_0+\Delta\bigr)+\theta\bigl(x_1+\Delta\bigr)|\\
	&<(1-\theta)|x_0+\Delta|+\theta|x_1+\Delta|.
	\end{align*}
	Using this inequality, \eqref{MaxFormula} implies
	$$
	\Phi_{\gamma}(s,x_\theta)<(1-\theta)\Phi_\gamma(s,x_0)+\theta\Phi_\gamma(s,x_1).
	$$
	This gives the strict convexity of $\Phi_\gamma(s,\cdot)$. If $\partial_x\Phi_\gamma(s,x_0)=\partial_x\Phi_{\gamma}(s,x_1)$ for two distinct $x_0$ and $x_1,$ then $\Phi_\gamma(s,x)=cx+d$ for any $x$ between $x_0$ and $x_1,$ where $c,d$ are constants. However, this contradicts the strict convexity of $\Phi_\gamma(s,\cdot).$
\end{proof}

Lemma \ref{lem9} establishes the strict convexity of the PDE solution $\Phi_{\gamma}$ in $\gamma.$

\begin{lemma}
	\label{lem9}
	Let $\gamma_0,\gamma_1\in\mathcal{U}$, $x_0,x_1\in\mathbb{R}$, and $\theta\in(0,1).$ Recall the notations $x_\theta $ and $\gamma_\theta$ from \eqref{eq-2}. If $\gamma_0\neq \gamma  _1$, we have
	\begin{align*}
	\Phi_{\gamma_\theta}(0,x_\theta )&< (1-\theta)\Phi_{\gamma_0}(0,x_0)+\theta\Phi_{\gamma_1}(0,x_1).
	\end{align*}
\end{lemma}

\begin{proof} We adapt a slightly simplified version of the argument of \cite[Theorem 4]{AC14}.
	Suppose that $x_0,x_1\in\mathbb{R}$, $\gamma  _0\neq \gamma  _1$ and $\theta\in(0,1)$. Using the right-continuity of $\gamma_0,\gamma_1,$ without loss of generality, we may assume that there exist some $0<r<s<1$ such that
	\begin{align}
	\label{proof:thm1:eq2}
	\mbox{$\gamma_0>\gamma_1$ on $[r,s]$}.
	\end{align}
	First we claim that
	\begin{align}
	\label{lem9:proof:eq1}
	\Phi _{\gamma_\theta}(r,x_\theta )&<(1-\theta)\Phi_{\gamma_0}(r,x_0)+\theta\Phi_{\gamma_1}(r,x_1).
	\end{align}
	Suppose, on the contrary, that equality holds. Recall the notations used in the proof of Lemma \ref{lem4}. Let $u_{\gamma_0},u_{\gamma_\theta},u_{\gamma_1}$ be the corresponding maximizers of \eqref{eq16} generated by \eqref{max}. From Lemma \ref{lem4}, 
	\begin{align}
	\begin{split}
	\label{eq12}
	F_\theta^{r,s}(u_{\gamma_\theta},x_\theta )&\leq (1-\theta) F_0^{r,s}(u_{\gamma_\theta},x_0)+ \theta F_1^{r,s}(u_{\gamma_\theta},x_1).
	\end{split}
	\end{align}
	Note that
	\begin{align*}
	F_0^{r,s}(u_{\gamma_\theta},x_0)\leq \Phi_{\gamma_0}(r,x_0),\,\,F_1^{r,s}(u_{\gamma_\theta},x_1)\leq \Phi_{\gamma_1}(r,x_1),\,\,
	F_\theta^{r,s}(u_{\gamma_\theta},x_\theta )=\Phi_{\gamma_\theta}(r,x_\theta ).
	\end{align*}
	Consequently, from \eqref{eq12} and the assumption that
	$$
	\Phi_{\gamma_\theta}(r,x_\theta )=(1-\theta)\Phi_{\gamma_0}(r,x_0)+\theta\Phi_{\gamma_1}(r,x_1),
	$$
	we obtain
	\begin{align*}
	F_0^{r,s}(u_{\gamma_\theta},x_0)=\Phi_{\gamma_0}(r,x_0),\\
	F_1^{r,s}(u_{\gamma_\theta},x_1)=\Phi_{\gamma_1}(r,x_1).
	\end{align*}
	In other words, $u_{\gamma_\theta}$ realizes the maximum of the representations for $\Phi_{\gamma_0}(r,x_0)$ and $\Phi_{\gamma_1}(r,x_1)$. Now from \eqref{proof:thm1:eq2} and Lemma \ref{prop1}, we conclude $u_{\gamma_0}=u_{\gamma_\theta}$ with respect to the metric $d_0$ defined in \eqref{metric}. Since $u_{\gamma_0}$ and $u_{\gamma_\theta}$ are continuous on $[r,s],$ we have
	\begin{align}
	\label{eq13}
	\partial_x\Phi_{\gamma_0}(w,X_{\gamma_0}(w))=u_{\gamma_0}(w)=u_{\gamma_\theta}(w)=\partial_x\Phi_{\gamma_\theta}(w,X_{\gamma_\theta}(w))
	\end{align}
	for all $r\leq w\leq s,$ where $(X_{\gamma_0}(w))_{r\leq w\leq s}$ and $(X_{\gamma_\theta}(w))_{r\leq w\leq s}$ satisfy respectively,
	\begin{align*}
	dX_{\gamma_0}(w)&=\gamma  _0(w)\xi''  (w)\partial_x\Phi_{\gamma_0}(w,X_{\gamma_0}(w))dr+\xi''  (w)^{1/2}dB(w),\\
	X_{\gamma_0}(w)&=x_0,\\
	dX_{\gamma_\theta}(w)&=\gamma  _\theta(w)\xi''  (w)\partial_x\Phi_{\gamma_\theta}(w,X_{\gamma_\theta}(w))dr+\xi''  (w)^{1/2}dB(w),\\
	X_{\gamma_\theta}(w)&=x_\theta .
	\end{align*}
	From \eqref{eq10}, \eqref{eq13} and It\^{o}'s isometry, we obtain
	\begin{align}\label{eq14}
	\partial_{xx}\Phi_{\gamma_0}(w,X_{\gamma_0}(w))=\partial_{xx}\Phi_{\gamma_\theta}(w,X_\theta(w)),\,\,\forall w\in[r,s].
	\end{align}
	Combined with \eqref{eq11}, this gives
	\begin{align}
	\begin{split}\label{eq9}
	\int_r^s\gamma  _0(w)\xi''  (w)\e(\partial_{xx}\Phi_{\gamma_0}(w,X_0(w)))^2dw
	&=\int_r^s\gamma  _\theta(w)\xi''  (w)\e(\partial_{xx}\Phi_{\gamma_\theta}(w,X_\theta(w)))^2dw.
	\end{split}
	\end{align}
	For any $r< w<s,$ neither $\e(\partial_{xx}\Phi_{\gamma_0}(w,X_{\gamma_0}(w)))^2$ nor $\e(\partial_{xx}\Phi_{\gamma_\theta}(w,X_{\gamma_\theta}(w)))^2$ is equal to zero. In fact, from Girsanov's theorem \cite[Theorem 5.1]{KS}, the distribution of $X_{\gamma_0}(w)$ is Gaussian with some change of measure and therefore is supported on $\mathbb{R}$. If, for instance,  $$\e(\partial_{xx}\Phi_{\gamma_0}(w,X_{\gamma_0}(w)))^2=0,$$ then $\partial_{xx}\Phi_{\gamma_0}(w,\cdot)=0,$ which contradicts Lemma \ref{lem3}. Now, from \eqref{proof:thm1:eq2} and \eqref{eq14},
	$$
	\int_r^{s}\gamma  _0(w)\xi''  (w)\e(\partial_{xx}\Phi_{\gamma_0}(w,X_{\gamma_0}(w)))^2dr>\int_r^s\gamma  _\theta(w)\xi''  (w)\e(\partial_{xx}\Phi_{\gamma_\theta}(w,X_{\gamma_\theta}(w)))^2dr.
	$$
	This contradicts \eqref{eq9}, so the inequality \eqref{lem9:proof:eq1} holds. This completes the claim \eqref{lem9:proof:eq1}.
	\smallskip
	
	Next, from the above claim, for any $y_0,y_1\in\mathbb{R}$,
	\begin{align*}
	\Phi_{\gamma_\theta}(w,y_\theta)<(1-\theta) \Phi_{\gamma_0}(w,y_0)+\theta\Phi_{\gamma_1}(w,y_1).
	\end{align*}
	Therefore,
	\begin{align*}
	C_\theta^{0,s}(u_{\gamma_\theta},x_\theta )&
	<(1-\theta)C_0^{0,s}(u_{\gamma_\theta},x_0)+\theta C_1^{0,s}(u_{\gamma_\theta},x_1).
	\end{align*}
	This together with
	$$
	L_\theta^{0,s}(u_{\gamma_\theta})=(1-\theta) L_0^{0,s}(u_{\gamma_\theta})+\theta L_1^{0,s}(u_{\gamma_\theta})
	$$
	gives
	\begin{align*}
	F_\theta^{0,s}(u_{\gamma_\theta},x_\theta )<(1-\theta)F_0^{0,s}(u_{\gamma_\theta},x_0)+\theta F_1^{0,s}(u_{\gamma_\theta},x_1).
	\end{align*}
	Since $u_{\gamma_\theta}$ is the maximizer for $\Phi_{\gamma_\theta}(0,x_\theta ),$ clearly
	\begin{align*}
	\Phi_{\gamma_\theta}(0,x_\theta )&=F_\theta^{0,t}(u_{\gamma_\theta},x_\theta )<(1-\theta) \Phi_{\gamma_0}(0,x_0)+\theta\Phi_{\gamma_1}(0,x_1).
	\end{align*}
	This completes our proof.
\end{proof}

\begin{proof}[\bf Proof of Theorem \ref{thm4}] Note that $\int_0^1s\xi''(s)\gamma(s)ds$ is linear in $\gamma.$ From the strict convexity of $\Phi_\gamma$ in Lemma \ref{lem9}, Theorem \ref{unique} follows.
\end{proof}

\subsection{Optimality}

In this subsection, we establish some consequences of the optimality of the Parisi measure. These will be of great use when we turn to the proof of disorder chaos. Our first main results are Theorem \ref{thm3.1} and Proposition \ref{prop3.2} below.

\begin{theorem}\label{thm3.1}
	$\gamma_{h} $ is not identically equal to zero.
\end{theorem}


 Let $\mu_{h} $ be the measure induced by $\gamma_{h} .$ From Theorem~\ref{thm3.1}, it is clear that $\Omega_h:=\mbox{supp}\mu_{h} $ is not empty. This fact will play an essential role throughout the rest of the paper. Next, we formulate some consistency conditions for $\gamma_{h} .$

\begin{proposition}\label{prop3.2}
	For any $s\in \Omega_h$, we have that
	\begin{align}
	\begin{split}
	\label{prop3.2:eq1}
	\e\bigl(\partial_x\Phi_{\gamma_{h} }(s,X_{\gamma_{h} }(s))\bigr)^2&=s,
	\end{split}\\
	\begin{split}
	\label{prop3.2:eq2}
	\xi''(s)\e\bigl(\partial_{xx}\Phi_{\gamma_{h} }(s,X_{\gamma_{h} }(s))\bigr)^2&\leq 1,
	\end{split}
	\end{align}
	where $X_{\gamma_{h} }=(X_{\gamma_{h} }(w))_{0\leq w<1}$ satisfies the following stochastic differential equation with initial condition $X_{\gamma_h}(0)=h.$
	\begin{align*}
	dX_{\gamma_{h} }(w)&=\xi''(w)\gamma_{h} (w)\partial_x\Phi_{\gamma_{h} }(w,X_{\gamma_{h} }(w))dw+\xi''(w)^{1/2}dW(w).
	\end{align*}
\end{proposition}

Recall from Subsection \ref{unique} that $\mathcal{P}_h$ defines a strictly convex functional on $\mathcal{U}$. Our approach to proving Theorem \ref{thm3.1} and Proposition \ref{prop3.2} relies on the directional derivative of the Parisi functional derived in Proposition \ref{prop3.1}.

\begin{proposition}\label{prop3.1}
	Let $\gamma_0,\gamma\in\mathcal{U}$. For any $\theta\in[0,1]$, define
	$
	\gamma_\theta=(1-\theta)\gamma_0+\theta\gamma.
	$
    Then
	\begin{align*}
	\frac{d\mathcal{P}_h(\gamma_\theta)}{d\theta}\Big|_{\theta=0}&=\frac{1}{2}\int_0^1\xi''(s)\bigl(\gamma(s)-\gamma_0(s)\bigr)\bigl(\e u_{\gamma_0}(s)^2-s\bigr)ds,
	\end{align*}
	where $d\mathcal{P}_h(\gamma_\theta)/d\lambda|_{\theta=0}$ is the right-derivative at $0.$ In addition, the following statements are equivalent:
	\begin{itemize}
		\item[$(i)$] $\gamma_0$ is the Parisi measure,
		\item[$(ii)$] $\frac{d\mathcal{P}_h(\gamma_\theta)}{d\theta}\Big|_{\theta=0}\geq 0$ for any $\gamma\in\mathcal{U}.$
	\end{itemize}
\end{proposition}

An analogue of Proposition \ref{prop3.1} was established in \cite[Theorem 2]{C14} for the Parisi functional in the formulation of the Parisi formula for the limiting free energy. One may find that exactly the same argument also applies to prove Proposition \ref{prop3.1}, so we once again do not reproduce the proof here.

\begin{proof}[\bf Proof of Theorem \ref{thm3.1}]
	We argue by contradiction. Assume that $\gamma_{h} \equiv 0$. Then the PDE \eqref{pde} reduces to a heat equation, in which case, one can solve
	\begin{align*}
	\Phi_{\gamma_{h} }(s,x)&=\e |x+\sqrt{\xi'(1)-\xi'(s)}Z|,
	\end{align*}
	where $Z$ is a standard normal random variable. Thus, $$
	\partial_x\Phi_{\gamma_h}(s,x)=\e\mbox{sign}\bigl(x+\sqrt{\xi'(1)-\xi'(s)}Z\bigr)
	$$
	and
	$$
	\e u_{\gamma_{h} }(s)^2=\e_{Z'} \bigl(\e_Z \mbox{sign}\bigl(h+\sqrt{\xi'(s)}Z'+\sqrt{\xi'(1)-\xi'(s)}Z\bigr)\bigr)^2,
	$$
	where $Z'$ is standard normal independent of $Z$ and $\e_Z$ and $\e_{Z'}$ are the expectations with respect to $Z$ and $Z'$, respectively.
	We now compute the derivative of $\e u_{\gamma_{h} }(s)^2$. To this end, define
	\begin{align*}
	f(s)&=\e_{Z'} g(s,h+\sqrt{s}Z')^2,\\
	g(s,x)&=\e_Z\mbox{sign}(x+\sqrt{1-s}Z).
	\end{align*}
	Observe that \begin{align*}
	g(s,x)&=1-2\p\bigl(x+\sqrt{1-s}Z<0\bigr)\\
	&=1-2\int_{-\infty}^0\frac{1}{\sqrt{2\pi(1-s)}}\exp\Bigl(-\frac{(z-x)^2}{2(1-s)}\Bigr)dz\\
	&=1-2\int_0^\infty\frac{1}{\sqrt{2\pi(1-s)}}\exp\Bigl(-\frac{(z+x)^2}{2(1-s)}\Bigr)dz,
	\end{align*}
	from which one can directly check that $$
	\partial_sg=-\frac{1}{2}\partial_{xx}g,\,\,\forall(s,x)\in[0,1)\times\mathbb{R}.
	$$
	Set $Z_{s}=h+\sqrt{s}Z'.$ Then
	\begin{align*}
	f'(s)&=2\e\partial_sg(s,Z_s)g(s,Z_s)+\frac{1}{\sqrt{s}}\e Z'\partial_xg(s,Z_s)g(s,Z_s)\\
	&=2\e\partial_sg(s,Z_s)g(s,Z_s)+\e \partial_{xx}g(s,Z_s)g(s,Z_s)+\e \bigl(\partial_xg(s,Z_s)\bigr)^2\\
	&=\e \bigl(\partial_xg(s,Z_s)\bigr)^2.
	\end{align*}
	To compute the last term, note that
	\begin{align*}
	\partial_xg(s,x)&=2\int_0^\infty \frac{(z+x)}{\sqrt{2\pi}(1-s)^{3/2}}\exp\Bigl(-\frac{(z+x)^2}{2(1-s)}\Bigr)dz\\
	&=-\frac{2}{\sqrt{2\pi(1-s)}}\exp\Bigl(-\frac{(z+x)^2}{2(1-s)}\Bigr)\Big|_{0}^\infty\\
	&=\frac{2}{\sqrt{2\pi(1-s)}}\exp\Bigl(-\frac{x^2}{2(1-s)}\Bigr).
	\end{align*}
	This gives
	\begin{align*}
	\e \bigl(\partial_xg(s,Z_s)\bigr)^2&=\frac{2}{\pi (1-s)}\frac{1}{\sqrt{2\pi s}}\int_{-\infty}^{\infty} \exp\Bigl(-\frac{z^2}{(1-s)}-\frac{(z-h)^2}{2s}\Bigr)dz.
	\end{align*}
	Here,
	\begin{align*}
	2sz^2+(1-s)(z-h)^2&=2sz^2+(1-s)z^2-2h(1-s)z+(1-s)h^2\\
	&=(1+s)z^2-2h(1-s)z+(1-s)h^2\\
	&=(1+s)\Bigl(z-\frac{(1-s)}{1+s}h\Bigr)^2-\frac{(1-s)^2}{1+s}h^2+(1-s)h^2\\
	&=(1+s)\Bigl(z-\frac{(1-s)}{1+s}h\Bigr)^2+\frac{2s(1-s)}{1+s}h^2.
	\end{align*}
	Therefore,
	\begin{align}
	\begin{split}
	\label{thm3.1:proof:eq1}
	f'(s)&=\e \bigl(\partial_xg(s,Z_s)\bigr)^2\\
	&=\frac{2}{\pi (1-s)}\frac{1}{\sqrt{2\pi s}}\int_{-\infty}^{\infty} \exp\Bigl(-\frac{(1+s)}{2s(1-s)}\Bigl(z-\frac{(1-s)}{1+s}h\Bigr)^2\Bigr)dz\exp\Bigl(-\frac{h^2}{1+s}\Bigr)\\
	&=\frac{2}{\pi\sqrt{1-s^2}}\exp\Bigl(-\frac{h^2}{1+s}\Bigr).
	\end{split}
	\end{align}
	To compute the derivative of $\e u_{\gamma_{h} }(s)^2$, we write
	\begin{align*}
	\e u_{\gamma_{h} }(s)^2&=\e_{Z'} \bigl(\e_Z \mbox{sign}\bigl(h/\sqrt{\xi'(1)}+\sqrt{\xi'(s)/\xi'(1)}Z'+\sqrt{1-\xi'(s)/\xi'(1)}Z\bigr)\bigr)^2.
	\end{align*}
	Using \eqref{thm3.1:proof:eq1} and the chain rule, we obtain
	\begin{align*}
	\bigl(\e u_{\gamma_{h} }(s)^2\bigr)'&=\frac{2\xi''(s)}{\pi\sqrt{\xi'(1)^2-\xi'(s)^2}}\exp\Bigl(-\frac{h^2}{\xi'(1)+\xi'(s)}\Bigr).
	\end{align*}
	Finally, from this formula,
	\begin{align*}
	\e u_{\gamma_{h} }(s)^2-s&=(1-\e u_{\gamma_{h} }(1)^2)-(s-\e u_{\gamma_{h} }(s)^2)\\
	&=\int_s^1\bigl(w-\e u_{\gamma_{h} }(w)^2\bigr)'dw\\
	&=\int_s^11-\bigr(\e u_{\gamma_{h} }(w)^2\bigr)'dw\\
	&<0,
	\end{align*}
	whenever $s$ is sufficiently close to $1.$ However, for any $\gamma\in\mathcal{U}$ with $\gamma=0$ on $[0,s_1]$ for $s_1$ being arbitrarily close to $1$, we have
	\begin{align*}
	\frac{d}{d\lambda}\mathcal{P}_h((1-\lambda)\gamma_{h} +\lambda \gamma)\Big|_{\lambda=0}=\frac{1}{2}\int_{s_1}^1\xi''(s)\gamma(s)\bigl(\e u_{\gamma_{h} }(s)^2-s\bigr)ds<0,
	\end{align*}
	which contradicts Proposition \ref{prop3.1}$(ii)$.
\end{proof}

\begin{proof}[\bf Proof of Proposition \ref{prop3.2}] Using the minimality of $\gamma_{h} $, one can vary $\gamma$ in Proposition \ref{prop3.1} to deduce \eqref{prop3.2:eq1}, while the validity of \eqref{prop3.2:eq2} is assured by \eqref{prop3.2:eq1} and the equation \eqref{eq10}. For detail, we refer the readers to \cite[Proposition 1]{C14} for an identical proof.
\end{proof}

\section{Establishing disorder chaos}
\label{sec:disorder_chaos}
Chaos in disorder at positive temperature has been intensively studied in recent years, see \cite{chatt,C12,C15,C14}. One of the major approaches is to control the coupled free energy with overlap constraint via a two-dimensional version of the Guerra-Talagrand (GT) inequality. Lately, this method was pushed forward to show chaos in disorder at zero temperature in \cite{ArnabChen15} for the spherical mixed even $p$-spin models. In this section, we utilize the ideas from \cite{C14,ArnabChen15} to establish Theorem~\ref{thm-1}.

\subsection{Guerra-Talagrand inequality}

We begin by recalling the GT inequality for the coupled free energy in terms of Parisi's PDE from \cite{C14}. Let 
\begin{align}\label{add:eq1}
V_N&:=\{R(\sigma^1,\sigma^2):\sigma^1,\sigma^2\in\Sigma_N\}.
\end{align} 
Let $q\in V_N$ be fixed. For $\beta>0$, define the coupled free energy as
\begin{align*}
F_{N,\beta}(q)&:=\frac{1}{N\beta}\e\log \sum_{R(\sigma^1 ,\sigma^2)=q}\exp \beta\bigl(H_{N,t}^1(\sigma^1 )+H_{N,t}^2(\sigma^2  )\bigr).
\end{align*}
Denote $\iota=1$ if $q\geq 0$ and $\iota=-1$ if $q<0.$ Define
\begin{align*}
T(s)&=\left[
\begin{array}{cc}
\xi''(s)&\iota t\xi''(s)\\
\iota t\xi''(s)&\xi''(s)
\end{array}\right]
\end{align*}
for $s\in[0,|q|)$ and
\begin{align}\label{eq-3}
T(s)&=\left[
\begin{array}{cc}
\xi''(s)&0\\
0&\xi''(s)
\end{array}\right]
\end{align}
for $s\in[|q|,1].$ Let $\mathcal{M}$ be the collection of all functions $\alpha$ on $[0,1]$ induced by a probability measure $\mu$, i.e., $\alpha(s)=\mu([0,s]).$ For $\alpha\in\mathcal{M}$,
consider the following PDE,
\begin{align}\label{sec4:eq1}
\partial_s\Psi_{\alpha,\beta}&=-\frac{1}{2}\bigl(\bigl\la T,\triangledown^2\Psi_{\alpha,\beta}\bigr\ra+\beta\alpha\bigl\la T\triangledown \Psi_{\alpha,\beta},\triangledown \Psi_{\alpha,\beta}\bigr\ra\bigr)
\end{align}
for $(\lambda,s,\bx)\in\mathbb{R}\times[0,1)\times\mathbb{R}^2$ with boundary condition
\begin{align}\label{sec4:eq2}
\Psi_{\alpha,\beta}(\lambda,1,\bx)=\frac{1}{\beta}\log\bigl(\cosh(\beta x_1)\cosh(\beta x_2)\cosh(\beta\lambda)+\sinh(\beta x_1)\sinh(\beta x_2)\sinh(\beta\lambda)\bigr).
\end{align}
The GT inequality established in  \cite[Theorem 6]{C14} states that for any $\beta>0$ and $\alpha\in\mathcal{M}$,
\begin{align}
\begin{split}\label{GT}
F_{N,\beta}(q)&\leq \frac{2\log 2}{\beta}+\Psi_{\alpha,\beta}(\lambda,0,h,h)-\lambda q\\
&\quad-\Bigl(\int_0^1\beta\alpha(s)s\xi''(s)ds+t\int_0^{|q|}\beta\alpha(s)s\xi''(s)ds\Bigr).
\end{split}
\end{align}
Note that the form here is slightly different than the one appeared in \cite[Equation (45)]{C14} by a change of variable in $\beta.$

To use this bound in our case, we need to convert it into the form for the coupled maximal energy by sending the inverse temperature $\beta$ to infinity. For any $\gamma\in\mathcal{U}$ with $\gamma(1-)<\infty,$ letting $\alpha=\gamma/\beta$, from \eqref{sec4:eq1} and \eqref{sec4:eq2}, we obtain that in the limit $\beta\rightarrow\infty,$
\begin{align}
\begin{split}\label{sec4:eq3}
F_N(q)&:=\frac{1}{N}\e\max_{R(\sigma^1 ,\sigma^2  )=q}\bigl(H_{N,t}^1(\sigma^1 )+H_{N,t}^2(\sigma^2  )\bigr)\leq \Lambda(\lambda,\gamma,q),
\end{split}
\end{align}
where
\begin{align}
\label{add:eq3}
\Lambda(\lambda,\gamma,q) :=\Psi_\gamma(\lambda,0,h,h)-\lambda q-\Bigl(\int_0^1\gamma(s)s\xi''(s)ds+t\int_0^{|q|}\gamma(s)s\xi''(s)ds\Bigr).
\end{align}
Let
\begin{align}
\begin{split}\label{gf}
g(\lambda,\bx)&:=\max\bigl(x_1+x_2+\lambda,-x_1-x_2+\lambda,x_1-x_2-\lambda,-x_1+x_2-\lambda\bigr).
\end{split}
\end{align}
Here, $\Psi_\gamma$ is the weak solution to the following equation,
\begin{align*}
\partial_s\Psi_{\gamma}&=-\frac{1}{2}\bigl(\bigl\la T,\triangledown^2\Psi_{\gamma}\bigr\ra+\gamma\bigl\la T\triangledown \Psi_{\gamma},\triangledown \Psi_{\gamma}\bigr\ra\bigr)
\end{align*}
for $(\lambda,s,\bx)\in\mathbb{R}\times[0,1)\times\mathbb{R}^2$
with boundary condition
\begin{align}\label{eq1}
\begin{split}
\Psi_{\gamma}(\lambda,1,\bx)&=g(\lambda,\bx),
\end{split}
\end{align}
where $\triangledown^2\Psi_\gamma$ is the Hessian and $\triangledown \Psi_\gamma$ is the gradient of $\Psi_\gamma$ in $\bx$ in the classical sense.
With an approximation procedure, one may relax the condition $\gamma(1-)<\infty$ to obtain the validity of \eqref{sec4:eq3} for arbitrary $\gamma\in\mathcal{U}$. We comment that these results can be justified by a similar treatment as \cite{AC16} and Proposition \ref{pde:prop1} with some minor adjustments. The fact that $\Psi_\gamma$ is two-dimensional does not influence the argument in any essential way. 

\subsection{Determination of $q_{t,h}$}\label{sec4.2}

We first define the overlap constant $q_{h}$ as the minimum of $\Omega_h$. Recall that  $\mu_{h} $ is the measure induced by $\gamma_{h}$ and $\Omega_h$ is the support of $\mu_{h}.$ From Theorem \ref{thm3.1}, $\Omega_h\neq \emptyset$ and $q_{h} <1.$ From Proposition~\ref{prop3.2}, letting $s=q_{h}$ gives
\begin{align}
\begin{split}
\label{sec4.2:eq1}
\e\bigl( \partial_x\Phi_{\gamma_{h} }(q_{h} ,h+\chi)\bigr)^2&=q_{h} ,
\end{split}\\
\begin{split}
\label{sec4.2:eq2}
\xi''(q_{h} )\e\bigl(\partial_{xx}\Phi_{\gamma_{h} }(q_{h} ,h+\chi)\bigr)^2&\leq 1
\end{split}
\end{align}
for $\chi$ a centered Gaussian random variable with variance $\xi'(q_{h} ).$  From the first equation and Lemma~\ref{lem3}, $q_{h} >0$ if $h\neq 0.$ Define
\begin{align*}
\psi_t(s)&=\e \partial_x\Phi_{\gamma_{h} }(q_{h} ,h+\chi_1(s))\partial_x\Phi_{\gamma_{h} }(q_{h} ,h+\chi_2(s))
\end{align*}
for $s\in[-q_{h} ,q_{h} ],$ where $\chi_1(s)$ and $\chi_2(s)$ are jointly Gaussian with mean zero and covariance
\begin{align*}
&\e(\chi_1(s))^2=\e(\chi_2(s))^2=\xi'(q_{h} ),\\
&\e \chi_1(s)\chi_2(s)=t\xi'(s).
\end{align*}
Assume that $h=0$. If $q_{h} >0$, then $\psi_1$ has two fixed points $0$ and $q_{h} $, which contradicts \cite[Theorem 5]{C141}. Thus, $q_{h} =0$ if $h=0.$ 

The overlap constant $q_{t,h}$ of Theorem \ref{thm-1} is determined by the fixed point of $\psi_t$.
\begin{proposition}\label{prop4.2}
	For any $t\in[0,1],$ $\psi_t$ maps $[-q_h,q_h]$ into itself and has a unique fixed point $q_{t,h},$ where $q_{t,h}=0$ if $h=0$ and $q_{t,h}>0$ if $h>0.$ In addition, $t\mapsto q_{t,h}$ is continuous with $q_{1,h}=q_{h} .$
\end{proposition}

\begin{proof} By the Cauchy-Schwarz inequality and \eqref{sec4.2:eq1}, it is easy to see that $\psi_t$ maps $[-q_h,q_h]$ into itself. If $h=0,$ then $\psi_t$ is defined only at the origin $0$ and obviously $0$ is the unique fixed point of $\psi_t.$ For $h\neq 0,$ we use \eqref{sec4.2:eq1}, \eqref{sec4.2:eq2}, and \cite[Theorem 5]{C141} to obtain the existence and uniqueness of the fixed point of $\psi_t.$ Moreover, observe that $$
	\psi_t(0)=\bigl(\e\partial_x\Phi_\gamma(q_{h} ,h+\chi)\bigr)^2.
	$$
	From Lemma \ref{lem3}, we see that $\psi_t(0)>0$ and thus, $q_{t,h}>0.$ Finally, from \eqref{sec4.2:eq2} and the Cauchy-Schwarz inequality, using Gaussian integration by parts gives
	\begin{align*}
	\psi_t'(s)&=t\xi''(s)\e \partial_{xx}\Phi_{\gamma_{h} }(q_{h} ,h+\chi_1(s))\partial_{xx}\Phi_{\gamma_{h} }(q_{h} ,h+\chi_2(s))\\
	&<\xi''(q_{h} )\e\bigl(\partial_{xx}\Phi_{\gamma_{h} }(q_{h} ,h+\chi)\bigr)^2\\
	&\leq 1
	\end{align*}
	for $(s,t)\in(-q_h,q_h)\times (0,1).$
	From this, $(\psi_t(s)-s)'<0$.
	By the implicit function theorem, there exists a continuous function $q(t)$ on $(0,1)$ such that $\psi_t(q(t))=q(t).$ From the uniqueness of $q_{t,h}$, it follows that $q_{t,h}=q(t)$ is continuous on $(0,1).$ To obtain the continuity of $t\mapsto q_{t,h}$ at $0$ and $1.$ Denote $a=\limsup_{t\rightarrow 1-}q_{t,h}$. From the continuity of $\psi_t(s)$ in both variables $(s,t),$
	$$
	\psi_1(a)=\limsup_{t\rightarrow 1-}\psi_t(q_{t,h})=\limsup_{t\rightarrow 1-}q_{t,h}=a,
	$$
	which implies that $a=q_{1,h}$ by the uniqueness of $q_{1,h}.$ Similarly, one may argue that $\liminf_{t\rightarrow 1-}q_{t,h}=q_{1,h}$. Thus, $t\mapsto q_{t,h}$ is continuous at $1.$ Similarly, we can conclude that this function is also continuous at $0.$ The equation $q_{1,h}=q_{h}$ is valid directly by the definition of $\psi_1.$
\end{proof}

\subsection{Control of the GT bound}

Recall the GT bound from \eqref{sec4:eq3}. Although $F_N(q)$ is defined only on $V_N,$ the functional $\Lambda(\lambda,\gamma,q)$ is indeed well-defined on $\mathbb{R}\times\mathcal{U}\times[-1,1].$ This allows us to define $$
\Lambda(q)=\inf_{\lambda\in\mathbb{R},\gamma\in\mathcal{U}}\Lambda(\lambda,\gamma,q)
$$
for $q\in[-1,1].$ The main goal of this subsection is to get the following controls for $\Lambda(q).$
\begin{proposition}
	\label{sec4:prop1} Let $t\in(0,1).$ Recall the overlap constant $q_{t,h}$ from Proposition \ref{prop4.2}.
	\begin{itemize}
		\item[$(i)$] For any $q\in [-q_{h} ,q_{h} ]\setminus\{q_{t,h}\},$ $\Lambda(q)<2M(h)$,
		\item[$(ii)$] 	For any $q\notin [-q_{h} ,q_{h} ]$, we have $\Lambda(q)<2M(h)$.
	\end{itemize}
\end{proposition}

The remainder of the subsection is devoted to establishing Proposition \ref{sec4:prop1}. First of all, we adapt the stochastic optimal control representation for the PDE $\Psi_\gamma$ in a similar manner as Theorem \ref{thm0}.  Denote by $\mathbf{W}=\{\mathbf{W}(w)=({W}_1(w),{W}_2(w)),\mathscr{G}_w,0\leq w\leq 1\}$ a two-dimensional Brownian motion, where $(\mathscr{G}_w)_{0\leq w\leq 1}$ satisfies the usual condition (see Definition~ 2.25 \cite{KS}). For $0\leq r< s\leq 1,$ denote by $\mathcal{D}[r,s]$ the space of all two-dimensional progressively measurable processes $v=(v_1,v_2)$ with respect to $(\mathscr{G}_w)_{r\leq w\leq s}$ satisfying $
\sup_{r\leq w\leq s}|v_1(w)|\leq 1$ and $\sup_{r\leq w\leq s}|v_2(w)|\leq 1.$ Endow the space $\mathcal{D}[s,t]$ with the norm
\begin{align}\label{eq2}
\|v\|_{r,s}&=\Bigl(\e\int_r^s(v_1(w)^2+v_2(w)^2)dw\Bigr)^{1/2}.
\end{align}
Define a functional
\begin{align*}
\bF_\gamma^{r,s}(\lambda,v,\bx)&=\e\left[\bC_\gamma^{r,s}(\lambda,v,\bx)-\bL_\gamma^{r,s}(v)\right]
\end{align*}
for  $(\lambda,v,\bx)\in\mathbb{R}\times\mathcal{D}[r,s]\times\mathbb{R}^2,$ where
\begin{align*}
\bC_\gamma^{r,s}(\lambda,v,\bx)&:=\Psi_\gamma\Bigl(\lambda,s,\bx+\int_r^s\alpha_\gamma(w)T(w)v(w)dw+\int_r^sT(w)^{1/2}d\mathbf{W}(w)\Bigr),\\
\bL_\gamma^{r,s}(v)&:=\frac{1}{2}\int_r^s\alpha_\gamma(w)\left<T(w)v(w),v(w)\right>dw.
\end{align*}
The following is an analogue of Theorem \ref{thm0}.

\begin{theorem}\label{thm2}  For any $\gamma\in\mathcal{U},$
	\begin{align}
	\label{thm2:eq2}
	\Psi_\gamma(\lambda,s,\bx)&=\max\left\{\bF_\gamma^{r,s}(\lambda,v,\bx)\big|v\in \mathcal{D}[r,s]\right\}.
	\end{align}
	The maximum of \eqref{thm2:eq2} is attained by $v_{\gamma}(r)=\triangledown \Psi_\gamma(\lambda,r,\vX_\gamma(r))$, where the two-dimensional stochastic process $(\vX_\gamma(w))_{r\leq w\leq s}$ satisfies
	\begin{align*}
	d\vX_\gamma(w)&=\gamma(w)T(w)\triangledown \Psi_\gamma(\lambda,w,\vX_\gamma(w))dw+T(w)^{1/2}d\mathbf{W}(w),\\
	\vX_\gamma(r)&=\bx.
	\end{align*}
\end{theorem}

The proof of Theorem \ref{thm2} is a standard computation of It\^{o}'s formula. We invite the readers to check either \cite[Theorem 2]{AC16} or \cite[Theorem 5]{C14}. The following lemma listed some key properties of $\Psi_\gamma$ when $\lambda=0.$ It will be used in the proof of Proposition \ref{sec4:prop1}$(i).$

\begin{lemma}
	\label{sec4.3:lem1}
	Assume $|q|<1.$ Let $\gamma\in\mathcal{U}$. For any $r\in [|q|,1)$ and $\bx\in\mathbb{R}^2$,
	\begin{align}
	\begin{split}\label{sec4.3:lem1:eq1}
	\Psi_\gamma(0,r,\bx)&=\Phi_{\gamma}(r,x_1)+\Phi_{\gamma}(r,x_2),
	\end{split}
	\end{align}
	and $\Psi_\gamma(\lambda,r,\bx)$ is partially differentiable in $\lambda$ in the classical sense with
	\begin{align}
	\begin{split}\label{sec4.3:lem1:eq2}
	\partial_\lambda \Psi_\gamma(0,r,\bx)&=\partial_x\Phi_{\gamma}(r,x_1)\partial_x\Phi_{\gamma}(r,x_2).
	\end{split}
	\end{align}
\end{lemma}

The proof of Lemma \ref{sec4.3:lem1} utilizes the variational representations for $\Phi_\gamma(r,x_1),$ $\Phi_{\gamma}(r,x_2)$, and $\Psi_\gamma(0,r,\bx)$ introduced above. While the validity of \eqref{sec4.3:lem1:eq1} follows directly from these representations, the proof of the partial differentiability of $\Psi_\gamma(\lambda,r,\bx)$ in $\lambda$ is relatively difficult and depends on the following technical lemma concerning about the right differentiability of some functions defined in terms of maximization problems.

\begin{lemma}[{\cite[Lemma 2]{C14}}]
	\label{lem-1}
	Let $K$ be a metric space and $I$ be an interval with right open edge. Let $f$ be a real-valued function on $K\times I$ and $f_0(y)=\sup_{a\in K}f(a,y).$
	Suppose that there exists a $K$-valued continuous function $a(y)$ on $I$ su	ch that $f_0(y)=f(a(y),y)$ and $\partial_yf$ is continuous on $K\times I$, then $f_0$ is right-differentiable with derivative $\partial_yf(a(y),y)$ for all $y\in I$.
\end{lemma}

\begin{proof}[\bf Proof of Lemma \ref{sec4.3:lem1}]
	Assume that $|q|<1$ and $\gamma\in \mathcal{U}$ is fixed. Let $r\in[|q|,1]$ and $s=1$ be fixed. Recall the function $g$ from \eqref{gf}. From Theorem \ref{thm2} associated to the pair $(r,s)$, 
	\begin{align}\label{sec4.3:lem1:proof:eq1}
		\Psi_\gamma(\lambda,r,\bx)&=\max\left\{\bF_\gamma^{r,1}(\lambda,v,\bx)\big|v\in \mathcal{D}[r,1]\right\}.
		\end{align}
	From \eqref{eq-3} and \eqref{eq1},
		\begin{align*}
		\Psi_\gamma(\lambda,r,\bx)&=\max_{v\in \mathcal{D}[r,1]} \e\Bigl[g\Bigl(\lambda,\bx+\int_r^1\xi''\gamma  vdw+\int_r^1\sqrt{\xi''}d\mathbf{W}\Bigr)-\frac{1}{2}\int_r^1\xi''\gamma\la v,v\ra dw\Bigr].
	\end{align*}
    Since $g(0,\by)=|y_1|+|y_2|$, the foregoing equation becomes
	\begin{align*}
	\Psi_\gamma(\lambda,r,\bx)&=\max_{v=(v_1,v_2)\in \mathcal{D}[r,1]} \Bigl\{\e\Bigl[\Bigl|x_1+\int_r^1\xi''\gamma  v_1dw+\int_r^1\sqrt{\xi''}d{W}_1\Bigr|-\frac{1}{2}\int_r^1\xi''\gamma v_1^2dw\Bigr]\\
	&\qquad\qquad\qquad+\e\Bigl[\Bigl|x_2+\int_r^1\xi''\gamma  v_2dw+\int_r^1\sqrt{\xi''}d{W}_2\Bigr|-\frac{1}{2}\int_r^1\xi''\gamma v_2^2dw\Bigr]\Bigr\}\\
	&=\Phi_\gamma(r,x_1)+\Phi_\gamma(r,x_2)
	\end{align*}
	by applying Theorem \ref{thm0} with the same pair $(r,s).$ This gives \eqref{sec4.3:lem1:eq1}.
	
	To establish the partial differentiability of $\Psi_\gamma$ in $\lambda$, observe that $g$ is Lipschitz in $\lambda$ and it can be written as
	\begin{align*}
	g(\lambda,\by)&=\max\bigl(|y_1+y_2|+\lambda,|y_1-y_2|-\lambda\bigr).
	\end{align*}
	Note that
	\begin{align*}
	\partial_\lambda g(\lambda,\by)&=\mbox{sign}\bigl(2\lambda+|y_1+y_2|-|y_1-y_2|\bigr)=2I_{\{2\lambda+|y_1+y_2|-|y_1-y_2|>0\}}-1\,\,a.e.,
	\end{align*}
	where $\mbox{sign}(a):=1$ if $a\geq 0$ and $:=-1$ if $a<0$ and $I_{\{2\lambda+|y_1+y_2|-|y_1-y_2|>0\}}$ is the indication function on $\{2\lambda+|y_1+y_2|-|y_1-y_2|>0\}$. By the dominated convergence theorem, for any $v\in \mathcal{D}[r,1],$
		\begin{align*}
		\partial_\lambda \mathcal{F}_\gamma^{r,1}(\lambda,v,\bx)&=2\p\Bigl(2\lambda+\bigl|Y_{\gamma,1}^v+Y_{\gamma,2}^v\bigr|>\bigl|Y_{\gamma,1}^v-Y_{\gamma,2}^v\bigr|\Bigr)-1,
		\end{align*}
		where $\bY_\gamma^v=(Y_{\gamma,1}^v,Y_{\gamma,2}^v)$ is defined by
		$$
		\bY_\gamma^v=\bx+\int_r^1\gamma\xi''vdw+\int_r^1\sqrt{\xi''}d\bW.
		$$
	Here, observe that the equality 
	$$
	2\lambda+\bigl|Y_{\gamma,1}^v+Y_{\gamma,2}^v\bigr|=\bigl|Y_{\gamma,1}^v-Y_{\gamma,2}^v\bigr|
	$$
	implies that $Y_{\gamma,1}^v$ or $Y_{\gamma,2}^v$ must be equal to $\lambda$ or $-\lambda,$ which happens with zero probability. Consequently,
	\begin{align*}
	\p\Bigl(2\lambda+\bigl|Y_{\gamma,1}^v+Y_{\gamma,2}^v\bigr|=\bigl|Y_{\gamma,1}^v-Y_{\gamma,2}^v\bigr|\Bigr)=0.
	\end{align*}
	This guarantees that $\partial_\lambda \mathcal{F}_\gamma^{r,1}(\lambda,v,\bx)$ is continuous in $(\lambda,v)$.
    Next, we recall that from Theorem \ref{thm2}, the optimizer of \eqref{sec4.3:lem1:proof:eq1} is given by $v_\gamma^{r,\lambda,\bx}(w)=\triangledown\Psi_\gamma\bigl(\lambda,w,\bX_\gamma^{r,\lambda,\bx}(w)\bigr)$, where $\bX_\gamma^{r,\lambda,\bx}$ satisfies
	\begin{align*}
	d\bX_\gamma^{r,\lambda,\bx}(w)&=\gamma(w)\xi''(w)\triangledown \Psi_\gamma(\lambda,w,\bX_\gamma^{r,\lambda,\bx}(w))dw+\sqrt{\xi''(w)}d\mathbf{W}(w),\,\,r\leq w\leq 1,\\
	\bX_\gamma^{r,\lambda,\bx}(r)&=\bx.
	\end{align*}
    Since $g$ is uniformly Lipschitz, we can apply the same argument as Proposition \ref{pde:prop1} to obtain that $\triangledown \Psi_\gamma(\lambda,s,\bx)$ is Lipschitz in $(\lambda,\bx)\in \mathbb{R}^3$ uniformly over all $s\in [0,s_0]$ for any $0<s_0<1.$ As a result, an application of the Gronwall inequality yields (see, e.g., the argument of \cite[Theorem 3]{AC14}) that for any $s_0\in(0,1)$, there exists $K_0$ such that with probability $1,$
	\begin{align*}
	\sup_{w\in [0,s_0]}\bigl|\bX_{\gamma}^{r,\lambda,\bx}(w)-\bX_{\gamma}^{r,\lambda',\bx}(w)\bigr|\leq K_{0}|\lambda-\lambda'|,\,\,\forall \lambda,\lambda'\in \mathbb{R}.
	\end{align*}
    Therefore, the process $\lambda\mapsto v_\gamma^{r,\lambda,\bx}$ is continuous, i.e., $\lim_{\lambda'\rightarrow \lambda}\|v_\gamma^{r,\lambda',\bx}-v_\gamma^{r,\lambda,\bx}\|_{r,1}=0,$ where $\|\cdot\|_{r,1}$ is defined through \eqref{eq2}. This together with the continuity of $\partial_\lambda \mathcal{F}_\gamma^{r,1}(\lambda,v,\bx)$ in $(\lambda,v)$ established above implies that $\Psi_\gamma(\lambda,r,\bx)$ is right partially differentiable at all $\lambda\in \mathbb{R}$ by Lemma \ref{lem-1}. Furthermore, this right derivative is equal to  $\partial_\lambda \mathcal{F}_\gamma^{r,1}(\lambda,v_\gamma^{r,\lambda,\bx},\bx)$ and is continuous in $\lambda.$ It is a well-known result (see, e.g., \cite{BRU}) that if a function has a continuous right derivative on an open interval $U$, then it is differentiable on $U$. From these, we obtain the partial differentiability of $\Psi_\gamma(\lambda,r,\bx)$ in $\lambda.$
    
    To see how \eqref{sec4.3:lem1:eq2} is obtained, we note that when $\lambda=0,$ $g(0,\by)=|y_1|+|y_2|$, which implies $\triangledown \Psi_\gamma(0,w,\by)=(\partial_x\Phi_\gamma(w,y_1),\partial_x\Phi_\gamma(w,y_2)).$ Thus, the stochastic processes $X_{\gamma,1}^{r,0,\bx}$ and $X_{\gamma,2}^{r,0,\bx}$ are exactly the minimizers of the variational representations \eqref{MaxFormula} for $\Phi_\gamma(r,x_1)$ and $\Phi_\gamma(r,x_2)$ with respect to independent standard Brownian motions $W_1$ and $W_2$, respectively. Thus,
    \begin{align*}
    \partial_\lambda\Psi_\lambda(0,r,\bx)&=\e \mbox{sign}\bigl(X_{\gamma,1}^{r,0,\bx}\bigr)\cdot \e\mbox{sign}\bigl(X_{\gamma,2}^{r,0,\bx}\bigr)\\
    &=\partial_x\Phi_\gamma(r,x_1)\partial_x\Phi_\gamma(r,x_2),
    \end{align*}
    where the last equality is obtained by noting that $\partial_x\Phi_\gamma(r,x_j)=\e \mbox{sign}\bigl(X_{\gamma,j}^{r,0,\bx}\bigr)$, which can be derived by the above argument. This completes our proof.
\end{proof}

\begin{proof}[\bf Proof of Proposition \ref{sec4:prop1}$(i)$] If $h=0,$ from Subsection \ref{sec4.2} and Proposition \ref{prop4.2}, we see that $q_{h} =0,$ so $[-q_{h} ,q_{h} ]\setminus\{q_{t,h}\}=\emptyset$ and $(i)$ is vacuously valid. Assume now $h\neq 0$ and $q\in[-q_{h} ,q_{h} ]\setminus\{q_{t,h}\}.$ Recall the Gaussian random variables  $\chi,\chi_1(s),\chi_2(s)$ and the function $\psi_t$ introduced in Subsection~\ref{sec4.2}. Note that $\gamma_{h} =0$ on $[0,q_{h} )$ and $q_h<1.$ From \eqref{sec4.3:lem1:eq1},
	\begin{align*}
	\Psi_{\gamma_{h} }(0,0,h,h)&=\e\Psi_{\gamma_{h} }(0,q_{h} ,h+\chi_1(q),h+\chi_2(q))\\
	&=\e \Phi_{\gamma_{h} }(q_{h} ,h+\chi_1(q))+\e\Phi_{\gamma_{h} }(q_{h} ,h+\chi_2(q))\\
	&=2\e\Phi_{\gamma_{h} }(q_{h} ,h+\chi)\\
	&=2\Phi_{\gamma_{h} }(0,h).
	\end{align*}
    This implies $\Lambda(0,\gamma_h,q)=2\mathcal{P}_h(\gamma_h)=2M(h)$ 
	because $$\int_0^{|q|}\gamma_{h}(s)s\xi''(s)ds=0$$ in $\Lambda$ (see \eqref{add:eq3}) for $|q|\leq q_h.$ 
	On the other hand, from \eqref{sec4.3:lem1:eq2},
	\begin{align*}
	\partial_\lambda\Psi_{\gamma_{h} }(0,0,h,h)&=\e \partial_\lambda\Psi_{\gamma_{h} }(0,q_{h} ,h+\chi_1(q),h+\chi_2(q))\\
	&=\e\partial_x\Phi_{\gamma_{h} }(q_{h} ,h+\chi_1(q))\partial_x\Phi_{\gamma_{h} }(q_{h} ,h+\chi_2(q))\\
	&=\psi_t(q).
	\end{align*}
	Therefore,
	\begin{align*}
	\partial_{\lambda}\Lambda(0,\gamma_{h} ,q)&=\psi_t(q)-q.
	\end{align*}
	Since $\psi_t$ has a unique fixed point $q_{t,h}$ by Proposition \ref{prop4.2}, depending on the sign of $\psi_t(q)-q$, we may vary $\lambda$ slightly to obtain $\Lambda(q)\leq\Lambda(\lambda,\gamma_{h} ,q)<\Lambda(0,\gamma_{h} ,q)=2M(h).$
\end{proof}

To show Proposition \ref{sec4:prop1}$(ii)$, we need the following crucial lemma, which allows us to quantify the difference between the one-dimensional and two-dimensional Parisi PDEs  in an elementary way:
\begin{proposition}\label{sec4.3:prop1}
	Assume that $|q|>q_{h} .$ Define $\gamma_t\in\mathcal{U}$ by $$
	\gamma_t=\left\{
	\begin{array}{ll}
	\frac{\gamma_{h}}{1+t},&\mbox{on $[0,|q|)$},\\
	\gamma_h,&\mbox{on $[|q|,1)$}.
	\end{array}
	\right.
	$$ The following two statements hold:
	\begin{enumerate}
		\item[$(i)$] If $v_{\gamma_t}=(v_1,v_2)$ is the maximizer to the variational problem \eqref{thm2:eq2} for $\Psi_{\gamma_t}(0,0,\bx)$ using $(r,s)=(0,|q|),$ then
		\begin{align}
		\begin{split}
		\label{sec4.3:prop:eq1}
		\Psi_{\gamma_t}(0,0,\bx)&\leq \Phi_{\gamma_{h} }(0,x_1)+\Phi_{\gamma_{h} }(0,x_2)
		-\frac{t(1-t)}{2(1+t)^2}\int_{0}^{|q|}\gamma_{h} \xi''\e\left(v_1-\iota v_2\right)^2dw.
		\end{split}
		\end{align}
		
		\item[$(ii)$] Define
		\begin{align}
		\begin{split}\label{sec4.3:prop:eq2}
		\left(u_1(w),u_2(w)\right)&=\frac{1}{1+t}T(w)v_{\gamma_t}(w),
		\end{split}\\
		\begin{split}
		\label{sec4.3:prop:eq3}
		(B_1(r),{B}_2(r))&=\frac{1}{\xi''(w)^{1/2}}T(w)^{1/2}\mathbf{W}(w)
		\end{split}
		\end{align}
		for $0\leq r\leq |q|.$ If $$
		\Psi_{\gamma_t}(0,0,\bx)=\Phi_{\gamma_{h} }(0,x_1)+\Phi_{\gamma_{h} }(0,x_2),$$
		then $u_1$ and $u_2$ are the maximizers of the variational problems \eqref{MaxFormula} for $\Phi_{\gamma_{h} }(0,x_1)$ and $\Phi_{\gamma_{h} }(0,x_2)$ using $(r,s)=(0,|q|)$ with respect to the standard Brownian motions $B_1$ and $B_2$. Moreover, on the interval $[q_{h} ,|q|]$,
		\begin{align*}
		u_1(w)&=\partial_x\Phi_{\gamma_{h} }(w,X_{1,\gamma_h}(w)),\\
		u_2(w)&=\partial_x\Phi_{\gamma_{h} }(w,X_{2,\gamma_h}(w)),
		\end{align*}
		where $(X_{1,\gamma_h}(w))_{0\leq w\leq|q|}$ and $(X_{2,\gamma_h}(w))_{0\leq w\leq |q|}$ satisfy
		\begin{align*}
		dX_{1,\gamma_h}(w)&=\gamma_{h} (w)\xi''(w)\partial_x\Phi_{\gamma_{h} }(w,X_{1,\gamma_h}(w))dw+\xi''(w)^{1/2}dB_1(w),\\
		dX_{2,\gamma_h}(w)&=\gamma_{h}(w)\xi''(w)\partial_x\Phi_{\gamma_{h} }(w,X_{2,\gamma_h}(w))dw+\xi''(w)^{1/2}dB_2(w)
		\end{align*}
		with initial condition $X_{\gamma_h}(0)=x_1$ and $X_{\gamma_h}(0)=x_2.$
	\end{enumerate}
\end{proposition}

Theorem \ref{thm2} and Proposition \ref{sec4.3:prop1} are analogous to the results of Theorem 5 and Proposition 5 in \cite{C14}, which were used to control the PDE solution $\Psi_{\alpha,\beta}$ in the GT bound \eqref{GT}. Although the space $\mathcal{U}$ is different from $\mathcal{M}$ in \cite{C14}, the same proofs there carry through in our case.
We do not reproduce the proofs here.

\begin{proof}[\bf Proof of Proposition \ref{sec4:prop1}$(ii)$]
	Let $q\notin[-q_{h} ,q_{h} ].$ Observe that
	\begin{align*}
	\Lambda(q)&\leq\Lambda(0,\gamma_t,q)=\Psi_{\gamma_t}(0,0,h,h)-\int_0^1w\xi''(w)\gamma_{h} (w)dw.
	\end{align*}
	It suffices to show that $\Psi_{\gamma_t}(0,0,h,h)<2\Phi_{\gamma_{h} }(0,h).$ If, on the contrary, the two sides are the same, then Proposition \ref{sec4.3:prop1}$(ii)$ readily implies
	\begin{align*}
	u_1(w)&=\partial_x\Phi_{\gamma_{h} }(w,X_{1,\gamma_h}(w)),\\
	u_2(w)&=\partial_x\Phi_{\gamma_{h} }(w,X_{2,\gamma_h}(w)),
	\end{align*}
	for $w\in [q_{h} ,|q|]$, where $X_{1,\gamma_h}=(X_{1,\gamma_h}(w))_{0\leq w\leq |q|}$ and $X_{2,\gamma_h}=(X_{2,\gamma_h}(w))_{0\leq w\leq |q|}$ satisfy
	\begin{align}
	\begin{split}\label{sde}
	X_{1,\gamma_h}(w)&=h+\int_0^w\gamma_{h} (a)\xi''(a)\partial_x\Phi_{\gamma_{h} }(a,X_{1,\gamma_h}(a))da+\int_0^w\xi''(a)^{1/2}dB_1(a),\\
	X_{2,\gamma_h}(w)&=h+\int_0^w\gamma_{P}(a)\xi''(a)\partial_x\Phi_{\gamma_{h} }(a,X_{2,\gamma_h}(a))da+\int_0^w\xi''(a)^{1/2}dB_2(a).
	\end{split}
	\end{align}
	On the other hand, from \eqref{sec4.3:prop:eq1}, $v_1=\iota v_2$ on $[q_{h} ,|q|]$, which together with \eqref{sec4.3:prop:eq2} leads to $u_1=\iota u_2$ on $[q_{h} ,|q|].$ Now since $\partial_x\Phi_{\gamma_{h} }(w,\cdot)$ is a strictly increasing odd function by Lemma \ref{lem3},
	$$
	X_{1,\gamma_h}(w)=\iota X_{2,\gamma_h}(w),\,\,\forall w\in[q_{h} ,|q|].
	$$
	Consequently, from \eqref{sde}, 
	$$
	0=X_{1,\gamma_h}(w)-\iota X_{2,\gamma_h}(w)=(1-\iota)h+\int_0^w\xi''(a)^{1/2}d(B_1-\iota B_2)(a),\,\,\forall w\in[q_{h} ,|q|].
	$$
	This forces $B_1=\iota B_2$. However, from \eqref{sec4.3:prop:eq3}, $\iota w=\e B_1(w)B_2(w)=\iota tw$ for $w\in[q_{h} ,|q|]$, a contradiction since $t\in(0,1).$ This finishes our proof.
\end{proof}

\subsection{Proof of Theorem \ref{thm-1}}

Before we start the proof, we recall some concentration inequalities for the extreme values of Gaussian processes. They will be used several times in this subsection and Section \ref{Sec5}.   Recall the Hamiltonians, $H_N^h$, $H_{N,t}^{1,h}$, and $H_{N,t}^{2,h}$ from Section $1$. Observe that if $h=0,$ for any nonempty $A\subset\Sigma_N$ and $\hat{A}\subset \Sigma_N\times\Sigma_N$,
\begin{align*}
\max_{\sigma\in A}\e\Bigl(\frac{H_N^h(\sigma)}{N}\Bigr)^2&=\frac{\xi(1)}{N}
\end{align*}
and
\begin{align*}
\max_{(\sigma^1 ,\sigma^2  )\in \hat A}\e \Bigl(\frac{H_{N,t}^{1,h}(\sigma^1 )+H_{N,t}^{2,h}(\sigma^2  )}{N}\Bigr)^2&=\frac{2}{N}\max_{(\sigma^1 ,\sigma^2  )\in \hat{A}}\bigl(\xi(1)+t\xi(R(\sigma^1 ,\sigma^2  ))\bigr)\leq \frac{4\xi(1)}{N}.
\end{align*}
From the Gaussian concentration of measure (see \cite{BLM}), the random variables
\begin{align*}
{L_N^h}(A)&:=\max_{\sigma\in A}\frac{H_N^h(\sigma)}{N},\\
{\hat{L}}_{N,t}^h(\hat A)&:=\max_{R(\sigma^1 ,\sigma^2  )\in \hat A}\frac{H_{N,t}^{1,h}(\sigma^1 )+H_{N,t}^{2,h}(\sigma^2  )}{N}
\end{align*}
satisfy
\begin{align}
\label{con1}
\p\bigl(|L_N^h(A)-\e L_N^h(A)|\geq l\bigr)\leq 2\exp\Bigl(-\frac{Nl^2}{4\xi(1)}\Bigr).
\end{align}
and
\begin{align}
\label{con}
\p\bigl(|{\hat{L}}_{N,t}^h(\hat{A})-\e {\hat{L}}_{N,t}^h(\hat{A})|\geq l\bigr)\leq 2\exp\Bigl(-\frac{Nl^2}{8\xi(1)}\Bigr).
\end{align}
In addition, we note a well-known upper bound for the Gaussian extremal process (see, e.g., \cite{BLM}) that, for $h=0,$ gives
\begin{align}
\label{exg}
\p |L_N^0(A)|\leq \sqrt{2\xi''(1)\log 2}.
\end{align}
A crucial feature of these inequalities is that the upper bounds are valid independent of the choice of $A,\hat{A}$ and external field $h.$

\begin{proof}[\bf Proof of Theorem \ref{thm-1}] To prove Theorem \ref{thm-1}, let $t\in(0,1)$ be fixed. Recall $V_N$ from \eqref{add:eq1}. For any $\varepsilon>0,$ let $I:=[-1,1]\setminus(q_{t,h}-\varepsilon,q_{t,h}+\varepsilon)$ and $I_N:=V_N\cap I$. Define
	$$
	\hat{A}_N=\{(\sigma^1 ,\sigma^2  )\in\Sigma_N^2:R(\sigma^1 ,\sigma^2  )\in I_N\}
	$$
	and for $q\in V_N,$
	$$
	\hat{A}_{N,q}=\{(\sigma^1 ,\sigma^2  )\in\Sigma_N^2:R(\sigma^1 ,\sigma^2  )=q\}.
	$$
	Write
	\begin{align*}
	\e{\hat{L}}_{N,t}^h(\hat{A}_N)=\e\max_{q\in I_N}\bigr({\hat{L}}_{N,t}^h(\hat{A}_{N,q})-\e {\hat{L}}_{N,t}^h(\hat{A}_{N,q})+\e {\hat{L}}_{N,t}^h(\hat{A}_{N,q})\bigr).
	\end{align*}
	Since $I_N$ contains at most $2N+1$ elements, this and the uniform concentration \eqref{con} together imply
	\begin{align*}
	\limsup_{N\rightarrow\infty}\e {\hat{L}}_{N,t}^h(\hat{A}_N)&\leq\limsup_{N\rightarrow\infty}\max_{q\in I_N}\e {\hat{L}}_{N,t}^h(\hat{A}_{N,q})\\
	&\leq \lim_{N\rightarrow\infty}\max_{q\in I_N}\Lambda(q)\\
	&\leq \max_{q\in S}\Lambda(q).
	\end{align*}
	Note that one may use the variational representation \eqref{thm2:eq2} for $\Psi_\gamma$ to obtain the continuity of $\Lambda(\lambda,\gamma,q)$ (see \cite[Theorem 4]{C14}), which implies that $\Lambda(q)$ is upper semi-continuous on $[-1,1]$.
	Since $\Lambda(q)<2M(h)$ for all $q\in[-1,1]\setminus\{q_{t,h}\}$ by Proposition \ref{sec4:prop1}, $\max_{q\in S}\Lambda(q)<2M(h)$ and thus,
	$$
	\limsup_{N\rightarrow\infty}\frac{1}{N}\e\max_{R(\sigma^1 ,\sigma^2  )\in I_N}\bigl(H_{N,t}^{1,h}(\sigma^1 )+H_{N,t}^{2,h}(\sigma^2  )\bigr)=\limsup_{N\rightarrow\infty}\e {\hat{L}}_{N,t}^h(\hat{A}_N)<2M(h).
	$$
	Using \eqref{con} again, there exists a constant $K>0$ such that the following event holds  with probability at least $1-K\exp\bigl(-N/K\bigr)$,
	\begin{align}\label{add:eq4}
	\frac{1}{N}\max_{R(\sigma^1 ,\sigma^2  )\in I_N}\bigl(H_{N,t}^{1,h}(\sigma^1 )+H_{N,t}^{2,h}(\sigma^2  )\bigr)&<\frac{1}{N}\Bigl(\max_{\sigma^1 \in\Sigma_N}H_{N,t}^{1,h}(\sigma^1 )+\max_{\sigma^2  \in\Sigma_N}H_{N,t}^{2,h}(\sigma^2  )\Bigr).
	\end{align}
	Note that
	\begin{align*}
	H_{N,t}^{1,h}(\sigma_{t,h}^1 )=\max_{\sigma \in\Sigma_N}H_{N,t}^{1,h}(\sigma ),\,\,H_{N,t}^{2,h}(\sigma_{t,h}^2)=\max_{\sigma  \in\Sigma_N}H_{N,t}^{2,h}(\sigma  ).
	\end{align*}
	Assume that $h=0.$ Since $q_{t,0}=0,$ the set $I_N$ is symmetric with respect to the origin. As a result, if $|R(\sigma_{t,h}^1 ,\sigma_{t,h}^2)|\in I_N$, then $R(\sigma_{t,h}^1 ,\sigma_{t,h}^2)\in I_N,$ which violates \eqref{add:eq4}. This means that the event $|R(\sigma_{t,h}^1 ,\sigma_{t,h}^2)|\in I_N$ has probability at most $K\exp\bigl(-N/K\bigr)$, which completes the proof of Theorem~\ref{thm-1}$(i).$ Similarly, if $h\neq 0$ and $R(\sigma_{t,h}^1,\sigma_{t,h}^2)\in I_N,$ then this again violates \eqref{add:eq4} and we conclude that the event $R(\sigma_{t,h}^1 ,\sigma_{t,h}^2)\in I_N$ has probability at most $K\exp\bigl(-N/K\bigr)$. The property that $q_{t,h}\in(0,q_h)$ follows directly from Proposition \ref{prop4.2}. This finishes the proof of Theorem~\ref{thm-1}$(ii).$
\end{proof}

\section{Establishing exponential number of multiple peaks}\label{Sec5}

This section establishes the main result on the energy landscape of $H_N$. Subsection \ref{sec:prop-0} proves Proposition~\ref{prop-0} and Subsection~\ref{sec:thm1} verifies Theorem \ref{thm1} by using chaos in disorder.

\subsection{Proof of Proposition \ref{prop-0}}
\label{sec:prop-0}

We will prove Proposition \ref{prop-0}. We first establish some concentration inequalities for the Hamiltonian $H_N$ and the magnetization $m_N$ when they are evaluated at the maximizer of $L_N^h.$ Recall $M(h)$ from \eqref{add:eq2}. 

\begin{proposition}\label{sec5:prop1}
	$M$ is continuously differentiable on $\mathbb{R}$ with
	\begin{align*}
	M'(h)&=\partial_x\Phi_{\gamma_{h} }(0,h).
	\end{align*}
	In particular, $M'(0)=0.$
\end{proposition}

\begin{proof}
	This proof is exactly the same as that of  \cite[Theorem 1]{P08}. Let $\delta>0.$ Using the Parisi formula \eqref{parisi},
	\begin{align*}
	\frac{M(h+\delta)-M(h)}{\delta}&\leq \frac{\Phi_{\gamma_{h} }(0,h+\delta)-\Phi_{\gamma_{h} }(0,h)}{\delta}
	\end{align*}
    and
	\begin{align*}
	\frac{M(h)-M(h-\delta)}{\delta}&\geq \frac{\Phi_{\gamma_h}(0,h)-\Phi_{\gamma_h}(0,h-\delta)}{\delta}.
	\end{align*}
    Applying Proposition~\ref{pde:prop1}$(iii)$ yields
	\begin{align*}
	\limsup_{\delta\downarrow 0}\frac{M(h+\delta)-M(h)}{\delta}&\leq \partial_x\Phi_{\gamma_{h} }(0,h)
	\end{align*}
	and
	\begin{align*}
	\liminf_{\delta\downarrow 0}\frac{M(h)-M(h-\delta)}{\delta}&\geq \partial_x\Phi_{\gamma_{h} }(0,h),
	\end{align*}
	which completes the proof of the differentiability of $M$. Using Lemma \ref{lem3} gives $M'(0)=0.$ Finally, since $M$ is convex and differentiable on $\mathbb{R},$ it is a classical result in convex analysis (see, e.g.,  \cite[Theorem 25.3]{convex}) that these two assumptions automatically imply the continuity of $M'$ and thus, $M$ is continuously differentiable. This finishes our proof.
\end{proof}

Let $\sigma_h:=\mbox{Argmax}_{\sigma\in\Sigma_N}H_N^h(\sigma).$ In what follows, we show that the magnetization of $\sigma_h$ is concentrated around a constant.

\begin{proposition}\label{thm3}
	For any $\varepsilon>0$, there exists some $K>0$ such that
	\begin{align*}
	\p \Bigl(\Bigl|m_N(\sigma_h)-M'(h)\Bigr|\geq \varepsilon\Bigr)\leq K\exp \Bigl(-\frac{N}{K}\Bigr)
	\end{align*}
	for any $N\geq 1.$
\end{proposition}

\begin{proof}
	Let $\varepsilon>0$.
	Define
	\begin{align*}
	B_{N}^+&=\max_{m_N(\sigma)>M'(h)+\varepsilon}\frac{H_N^h(\sigma)}{N},\\
	B_{N}^-&=\max_{m_N(\sigma)<M'(h)-\varepsilon}\frac{H_N^h(\sigma)}{N}.
	\end{align*}
	Then for any $\lambda>0,$
	\begin{align*}
	\e B_{N}^+&\leq \e\max_{m_N(\sigma)>M'(h)+\varepsilon}\Bigl(\frac{H_N(\sigma)}{N}+(h+\lambda)m_N(\sigma)-\lambda M'(h)\Bigr)-\lambda \varepsilon\\
	&\leq \e L_{N}^{h+\lambda} -\lambda(M'(h)+\varepsilon).
	\end{align*}
	and
	\begin{align*}
	\e B_{N}^-&\leq \e\max_{m_N(\sigma)<M'(h)-\varepsilon}\Bigl(\frac{H_N(\sigma)}{N}+(h-\lambda) m_N(\sigma)+\lambda M'(h)\bigr)\Bigr)-\lambda \varepsilon\\
	&\leq \e L_N^{h-\lambda} -\lambda(\varepsilon-M'(h)).
	\end{align*}
	Passing to the limit, 
	\begin{align}\label{prop0:proof:eq1}
	\limsup_{N\rightarrow\infty}\e B_N^\pm&\leq M(h\pm\lambda)-\lambda(\varepsilon\pm M'(h)).
	\end{align}
	By Proposition \ref{sec5:prop1}, the $\lambda$-derivative at zero on the right-hand side of \eqref{prop0:proof:eq1} is equal to $-\varepsilon$. Thus, we can choose $\lambda$ small enough so that
	\begin{align*}
	\limsup_{N\rightarrow\infty}\e B_N^\pm&\leq M(h)-\varepsilon'.
	\end{align*}
	for some $\varepsilon'>0.$
	From the Gaussian concentration of measure \eqref{con1}, there exists a constant $K>0$ such that as long as $N$ is big enough, the following holds with probability at least $1-K\exp(-N/K)$,
	\begin{align*}
	B_N^\pm\leq  L_N(h) -\frac{\varepsilon'}{2}.
	\end{align*}
	Therefore, the probability that $\sigma_h$ satisfies $|m_N(\sigma_h)-M'(h)|> \varepsilon$ is at most $K\exp(-N/K),$ which finishes the proof.
\end{proof}

Recall the function $E(h)$ from \eqref{eq-5}. Next, we show that $H_N(\sigma_h)/N$ is essentially near the energy level $E(h).$

\begin{proposition}\label{prop2}
	Let $h\geq 0$. For any $\varepsilon>0$, there exists some $K>0$ such that
	\begin{align*}
	\p\Bigl(\Bigl|\frac{H_N(\sigma_h)}{N}-E(h)\Bigr|\geq \varepsilon\Bigr)\leq K\exp\Bigl(-\frac{N}{K}\Bigr)
	\end{align*}
	for all $N\geq 1.$
\end{proposition}

\begin{proof}
	From the definition of $\sigma_h$, we write
	\begin{align*}
	\frac{H_N(\sigma_h)}{N}-E(h)&=\Bigl(L_N^h-\e L_N^h\Bigr)\\
	&+\Bigl(\e L_N^h-M(h)\Bigr)\\
	&+h\big(M'(h)-m_N(\sigma_h)\bigr).
	\end{align*}
	Now the first term on the right-hand side can be controlled by applying the Gaussian concentration of measure to $L_N^h$. The second term can be handled by noting that $M(h)=\lim_{N\rightarrow\infty}\e L_N^h$. Control of the last term follows from Proposition \ref{thm3}.
\end{proof}

Next we develop an auxiliary lemma that will be used to show that $\lim_{h\rightarrow \infty}E(h)=0$ in Proposition \ref{prop-0}.
\begin{lemma}\label{lem1}
	If $m_N(\sigma^1)=k_1/N$ and $m_N(\sigma^2)=k_2/N$, then $$
	R(\sigma^1,\sigma^2)\geq \frac{1}{4}\Bigl(\frac{3k_1}{N}+\frac{3k_2}{N}-2\Bigr).
	$$
\end{lemma}

\begin{proof}
	Set
	\begin{align*}
	P_1=\{1\leq i\leq N:\sigma_i^1=1\},\\
	P_2=\{1\leq i\leq N:\sigma_i^2=1\}.
	\end{align*}
	Note that $|P_1^c|=(N-k_1)/2$ and $|P_2^c|=(N-k_2)/2$ and that
	\begin{align*}
	|P_1\cap P_2|&=\frac{1}{2}\bigl(|P_1|+|P_2|-|P_1\cap P_2^c|-|P_2\cap P_1^c|\bigr)\\
	&\geq \frac{1}{2}\bigl(|P_1|+|P_2|-|P_2^c|-|P_1^c|\bigr)\\
	&=\frac{1}{2}(k_1+k_2).
	\end{align*}
	Lemma \ref{lem1} then follows by
	\begin{align*}
	NR(\sigma^1,\sigma^2)&=|P_1\cap P_2|+|P_1^c\cap P_2^c|-|P_1\cap P_2^c|-|P_1^c\cap P_2|\\
	&\geq |P_1\cap P_2|-|P_2^c|-|P_1^c|\\
	&\geq \frac{1}{2}\Bigl(k_1+k_2-\frac{(N-k_1)}{2}-\frac{(N-k_2)}{2}\Bigr)\\
	&=\frac{1}{4}\bigl(3k_1+3k_2-2N\bigr).
	\end{align*}
\end{proof}

\begin{lemma}\label{lem2}
	$\lim_{h\rightarrow\infty}E(h)=0.$
\end{lemma}
\begin{proof}
	The argument contains three major steps. For any $\varepsilon>0$, define the set
	$$
		A_N(h,\varepsilon):=\bigl\{\sigma\in\Sigma_N:|m_N(\sigma)-M'(h)|<\varepsilon\bigr\}.
	$$
		First, we verify that $\lim_{h\rightarrow\infty }M'(h)=1.$
	Observe that using \eqref{con1} and \eqref{exg} with $h=0$ and $A=\Sigma_N$, we have that, with probability at least $1-Ke^{-KN}$,
	\begin{align*}
	-C\leq \max_{\sigma\in\Sigma_N}\frac{H_N(\sigma)}{N}< C,
	\end{align*}
		where $C,K>0$ are constants independent of $N.$
    This inequality implies that, with probability at least $1-Ke^{-KN},$
	\begin{align*}
	-\frac{C}{h}+\max_{\sigma\in\Sigma_N}m_N(\sigma)\leq \max_{\sigma\in\Sigma_N}\Bigl(\frac{H_N(\sigma)}{Nh}+m_N(\sigma)\Bigr)\leq \frac{C}{h}+\max_{\sigma\in\Sigma_N}m_N(\sigma).
	\end{align*} 
	Therefore, from the definition of $M(h)$ and the Borel-Cantelli lemma,
	\begin{align}
	\begin{split}\label{eq-6}
	\lim_{h\rightarrow\infty}\frac{M(h)}{h}&=\lim_{N\rightarrow\infty}\max_{\sigma\in\Sigma_N}m_N(\sigma)=1.
	\end{split}
	\end{align}
	On the other hand, from Proposition \ref{thm3}, a similar reasoning also gives that for any $\varepsilon>0,$
	\begin{align}
	\begin{split}\label{eq-7}
	\lim_{h\rightarrow\infty}\frac{M(h)}{h}&=\lim_{h\rightarrow\infty}\lim_{N\rightarrow\infty}\max_{A_N(h,\varepsilon)}\Bigl(\frac{H_N(\sigma)}{Nh}+m_N(\sigma)\Bigr)\\
	&=\lim_{h\rightarrow\infty}\lim_{N\rightarrow\infty}\max_{A_N(h,\varepsilon)} m_N(\sigma).
	\end{split}
	\end{align}
	Note that $M'(h)$ is a nondecreasing continuous function. If $\lim_{h\rightarrow\infty}M'(h)\neq 1$, then one can choose $\varepsilon$ small enough such that the above limits \eqref{eq-6} and \eqref{eq-7} contradict each other. Therefore, we conclude that $\lim_{h\rightarrow\infty}M'(h)=1.$
	
	Next, consider an auxiliary free energy, $$
	V_N(s,h)=\frac{1}{N\beta_h}\e\log\sum_{A_N(h,\varepsilon_h)}\exp\bigl( \beta_h\sqrt{s}H_N(\sigma)\bigr),
	$$
	where $\varepsilon_h$ is any positive function of $h$ satisfying $\lim_{h\rightarrow\infty}\varepsilon_h=0$ and $$
	\beta_h:=(1-M'(h)+\varepsilon_h)^{-1/2}.$$ Using Gaussian integration by parts, we get
	\begin{align*}
	\partial_sV_N(s,h)&=\frac{\beta_h}{2}\e\bigl\la \bigl(\xi(1)-\xi(R(\sigma^1,\sigma^2))\bigr)\bigr\ra'\\
	&\leq \frac{\beta_h}{2}\max\bigl\{\xi(1)-\xi(R(\sigma^1,\sigma^2)):\sigma^1,\sigma^2\in A_N(h,\varepsilon_h)\bigr\}\\
	&\leq \frac{\beta_h\xi''(1)}{2}\Bigl(1-\frac{1}{2}\Bigl(3M'(h)-3\varepsilon_h-1\Bigr)\Bigr)\\
	&=\frac{3\xi''(1)}{4}\bigl(1-M'(h)+\varepsilon_h\bigr)^{1/2}.
	\end{align*}
	Here, $\sigma^1$ and $\sigma^2$ in the first equality are two i.i.d. samplings from the Gibbs expectation with respect to the partition function,
	$$
	\sum_{A_N(h,\varepsilon_h)}\exp\bigl( \beta_h\sqrt{s}H_N(\sigma)\bigr)
	$$
	and $\la\cdot\ra'$ denotes the corresponding Gibbs average. The second inequality also used Lemma \ref{lem1}.
	In conclusion, we obtain that
	\begin{align}
	\begin{split}\label{eq4}
	|V_N(1,h)-V_N(0,h)|&\leq \max_{0\leq s\leq 1}\partial_sV_N(s,h)\\
	&\leq \frac{3\xi''(1)}{4}\bigl(1-M'(h)+\varepsilon_h\bigr)^{1/2}.
	\end{split}
	\end{align}
	The two terms on the left-hand side have the limits given in equations \eqref{eq5} and \eqref{eq6}. Note that $\lim_{h\rightarrow\infty}\beta_h=\infty$. Using the evident inequalities
	\begin{align*}
	\e\max_{A_N(h,\varepsilon_h)}\frac{H_N(\sigma)}{N}\leq V_N(1,h)&\leq \frac{\log 2}{\beta_h}+\e\max_{A_N(h,\varepsilon_h)}\frac{H_N(\sigma)}{N},
	\end{align*}
	we deduce that
	\begin{align}\label{eq5}
	\limsup_{h\rightarrow\infty}\limsup_{N\rightarrow\infty}\e\max_{A_N(h,\varepsilon_h)}\frac{H_N(\sigma)}{N}=\limsup_{h\rightarrow\infty}\limsup_{N\rightarrow\infty}V_N(1,h).
	\end{align}
	In addition, since
	\begin{align*}
	0<V_N(0,h)&=\frac{1}{N\beta_h}\log |{A_N(h,\varepsilon_h)}|\leq\frac{\log 2}{\beta_h},
	\end{align*}
	we also have
	\begin{align}
	\label{eq6}
	\limsup_{h\rightarrow\infty}\limsup_{N\rightarrow\infty}|V_N(0,h)|=0.
	\end{align}
	As a result, plugging \eqref{eq5} and \eqref{eq6} into \eqref{eq4} yields
	\begin{align}\label{eq3}
	\limsup_{h\rightarrow\infty}\limsup_{N\rightarrow\infty}\e\max_{A_N(h,\varepsilon_h)}\frac{H_N(\sigma)}{N}=0.
	\end{align}
	
	Finally, observe that from Proposition \ref{thm3}, for any $\varepsilon,h>0$
	\begin{align*}
	M(h)=\lim_{N\rightarrow\infty}\e\max_{A_N(h,\varepsilon)}\Bigl(\frac{H_N(\sigma)}{N}+hm_N(\sigma)\Bigr).
	\end{align*}
	This implies that
	\begin{align*}
	|E(h)|&\leq \varepsilon h+\lim_{N\rightarrow\infty}\e\max_{A_N(h,\varepsilon)}\frac{H_N(\sigma)}{N}.
	\end{align*}
	Since \eqref{eq3} holds for any $\varepsilon_h$ with $\lim_{h\rightarrow \infty}\varepsilon_h=0,$ letting $\varepsilon=\varepsilon_h=1/h^2$ in the last inequality and sending $h\rightarrow\infty$, we finish the proof.
\end{proof}

\begin{proof}[\bf Proof of Proposition \ref{prop-0}]
	Clearly $E(0)=M(0)$. From Proposition \ref{prop2} and Lemma \ref{lem2}, $E$ is continuous and satisfies $\lim_{h\rightarrow\infty}E(h)=0$. To prove monotonicity, for any $h'>h\geq 0$, write
	\begin{align*}
	E(h')-E(h)&=\bigl(M(h')-M(h)-(h'-h)M'(h')\bigr)+h\bigl(M'(h)-M'(h')\bigr).
	\end{align*}
	Since $M$ is convex, the first bracket on the right-hand side is not positive and neither is the other term. Therefore, $E$ is nonincreasing.
\end{proof}

\subsection{Proof of Theorem \ref{thm1}}
\label{sec:thm1}
Recall that $q_h$ is the smallest number in the support of the Parisi measure $\mu_h$. Also, recall the constant $q_{t,h}\leq q_h$ from Proposition \ref{prop4.2}.

\begin{proof}[\bf Proof of Theorem \ref{thm1}]
	Fix $h\geq 0.$ For a given $\varepsilon>0$, fix $t\in(0,1)$ such that
	\begin{align}\label{eq-8}
	\bigl((1-\sqrt{t})+\sqrt{1-t}\bigr)\bigl(\sqrt{2\xi''(1)\log 2}+1\bigr)<\frac{\varepsilon}{2}
	\end{align}
	and
	\begin{align}\label{eq-9}
	0\leq q_{t,h}-q_h\leq \frac{\varepsilon}{2}.
	\end{align}
	Let $n\geq 2$ be an integer, which will be specified later. Consider $n$ i.i.d. copies  $H_N^1,\ldots, H_N^n$ of $H_N$. Set
	\begin{align*}
	H_{N,t}^{\ell,h}(\sigma^\ell)=\sqrt{t}H_N(\sigma^\ell)+\sqrt{1-t}H_N^\ell(\sigma^\ell)+h\sum_{i=1}^N\sigma_i^\ell.
	\end{align*}
	Denote by ${\sigma}_{t,h}^1,\ldots,{\sigma}_{t,h}^n$ the maximizers of $H_{N,t}^{1,h},\ldots,H_{N,t}^{n,h}$ over the configuration space $\Sigma_N$, respectively. We need the following control:
	\begin{enumerate}
		\item	From Theorem \ref{thm-1}, there exists some $K_1>0$ independent of $n$ such that for any $1\leq \ell<\ell'\leq n$, the following inequality holds,
		\begin{align*}
		\p\Bigl(\bigl|R\bigl({\sigma}_{t,h}^{\ell },{\sigma}_{t,h}^{\ell'}\bigr)-q_{t,h}\bigr|\geq \frac{\varepsilon}{2}\Bigr)\leq K_1\exp\Bigl(-\frac{N}{K_1}\Bigr),
		\end{align*}
		from which
		\begin{align}
		\begin{split}
		\label{proof:eq1}
		&\p\Bigl(\max_{1\leq \ell<\ell'\leq n}|R\bigl({\sigma}_{t,h}^{\ell },{\sigma}_{t,h}^{\ell'}\bigr)-q_{t,h}|\geq \frac{\varepsilon}{2}\Bigr)\\
		&\leq \sum_{1\leq \ell<\ell'\leq n}\p\Bigl(|R\bigl({\sigma}_{t,h}^{\ell },{\sigma}_{t,h}^{\ell'}\bigr)-q_{t,h}|\geq \frac{\varepsilon}{2}\Bigr)\\
		&\leq n^2K_1\exp\Bigl(-\frac{N}{K_1}\Bigr).
		\end{split}
		\end{align}
		\item From Proposition \ref{prop2}, there exists $K_2>0$ independent of $n$ such that
		\begin{align}\label{proof:eq2}
		&\p\Bigl(\max_{1\leq \ell\leq n}\Bigl|\frac{H_{N,t}^{\ell,0}(\sigma_{t,h}^\ell)}{N}-E(h)\Bigr|\geq \frac{\varepsilon}{2}\Bigr)\leq nK_2\exp\Bigl(-\frac{N}{K_2}\Bigr).
		\end{align}
		\item From Proposition \ref{thm3}, there exists $K_3>0$ independent of $n$ such that
		\begin{align}\label{proof:eq4}
				&\p\Bigl(\max_{1\leq \ell\leq n}\bigl|m_N(\sigma_{t,h}^\ell)-M'(h)\bigr|\geq \varepsilon\Bigr)\leq nK_3\exp\Bigl(-\frac{N}{K_3}\Bigr).
		\end{align}
		\item Note that
		\begin{align*}
		\Bigl|\frac{H_{N,t}^{\ell,0}(\sigma_{t,h}^\ell)}{N}-\frac{H_{N}(\sigma_{t,h}^\ell)}{N}\Bigr|&\leq (1-\sqrt{t})\frac{|L_N|}{N}+\sqrt{1-t}\frac{|L_N^\ell|}{N},
		\end{align*}
		where $L_N:=N^{-1}\max_{\sigma\in\Sigma_N}H_N(\sigma)$ and $L_N^{\ell}:=N^{-1}\max_{\sigma^\ell\in\Sigma_N}H_N^\ell(\sigma^\ell).$
		Here, from \eqref{exg}, for any $N\geq 1,$
		$$
		\e|L_N|\leq \sqrt{2\xi''(1)\log 2}.
		$$
		Using the Gaussian concentration of measure \eqref{con1} for $N^{-1}L_N,N^{-1}L_N^1,\ldots,N^{-1}L_N^n$ and the above inequality, there exists some $K_4>0$ independent of $n$ such that
		\begin{align}\label{proof:eq3}
		\p\Bigl(\frac{1}{N}\max\bigl(|L_N|,|L_{N}^1|,\ldots,|L_{N}^n|\bigr) \geq \sqrt{2\xi''(1)\log 2}+1\Bigr)\leq (n+1)K_4\exp\Bigl(-\frac{N}{K_4}\Bigr).
		\end{align}		
	\end{enumerate}	
	From above, we see that \eqref{proof:eq1}, \eqref{proof:eq2}, \eqref{proof:eq4}, and \eqref{proof:eq3} are valid with probability at least
	\begin{align*}
	&1-n^2K_1\exp\Bigl(-\frac{N}{K_1}\Bigr)-nK_2\exp\Bigl(-\frac{N}{K_2}\Bigr)-nK_3\exp\Bigl(-\frac{N}{K_3}\Bigr)-(n+1)K_4\exp\Bigl(-\frac{N}{K_4}\Bigr).
	\end{align*}
	Consequently,  the following three statements hold with probability at least $1-n^2K_5e^{-N/K_5}$ for $K_5:=K_1+K_2+K_3+K_4$. First, from \eqref{proof:eq2}, \eqref{proof:eq3}, and then \eqref{eq-8}, for $1\leq \ell\leq n$,	
	$$
	\Bigl|\frac{H_{N}(\sigma_{t,h}^\ell)}{N}-E(h)\Bigr|\leq \bigl((1-\sqrt{t})+\sqrt{1-t}\bigr)\bigl(\sqrt{2\xi''(1)\log 2}+1\bigr)+\frac{\varepsilon}{2}<\varepsilon.
	$$
	Second, for $1\leq \ell\leq n,$
	$$
	|m_N(\sigma_{t,h}^\ell)-M'(h)|<\varepsilon.
	$$
	Finally, from \eqref{eq-9} and \eqref{proof:eq1},  for $1\leq \ell<\ell'\leq n$,
	\begin{align*}
	\bigl|R(\sigma_{t,h}^\ell,\sigma_{t,h}^{\ell'})-q_{h}\bigr|\leq \frac{\varepsilon}{2}+|q_{t,h}-q_h|<\varepsilon.
	\end{align*}
	Note that the constants $K_1,K_2,K_3,K_4$ are independent of $n$. Now take $n$ to be the largest integer less than $e^{N/3K_5}$
	and set $S_{N}(h)=\{\sigma_{t,h}^1,\ldots,\sigma_{t,h}^n\}.$ Then the announced statement in Theorem \ref{thm1} holds with $K=3K_5$.
\end{proof}


\appendix
\setcounter{secnumdepth}{0}
\section{Appendix}


\noindent This appendix is devoted to establishing Proposition \ref{pde:prop1}. We prove a priori estimate first:

\begin{lemma}\label{pde:lem}
	Let $0<r_0<r_1<r_2<\infty.$ Suppose that $\kappa_1,\kappa_2\in L^\infty([r_0 ,r_2]\times\mathbb{R})$ and $g\in L^\infty(\mathbb{R})$ with $\|\kappa_i\|_\infty\leq C_i$ for $i=1,2$ and $\|g\|_\infty\leq C_0$. Assume that $u$ is the classical solution to
	$$
	\partial_ru(r,x)=\partial_{xx}u(r,x)+\kappa_1(r,x)\partial_xu(r,x)+k_2(r,x),\,\,\forall (r,x)\in(r_0 ,{r_2} ]\times\mathbb{R}
	$$
	with initial condition $u(r_0 ,x)=g(x).$ If $$
	r\mapsto\|\partial_xu(r,\cdot)\|_\infty$$
	is continuous on $[r_0 ,{r_2} ]$, then there exists a nonnegative continuous function $F$ on $[0,\infty)^3$ depending only on $r_0 ,r_2$ such that
	\begin{align*}
	\sup_{(r,x)\in[r_1 ,r_2]\times\mathbb{R}}|\partial_xu(r,x)|&\leq F(C_0,C_1,C_2).
	\end{align*}
	
\end{lemma}

\begin{proof} From the Duhamel principle,
	\begin{align*}
	u(r,x)&=P_{r-r_0 }g(x)+\int_{r_0 }^rP_{r-w}(\kappa_1(w,\cdot)\partial_xu(w,\cdot)+\kappa_2(w,\cdot))(x)dw,
	\end{align*}
	where for any $f\in L^\infty(\mathbb{R})$,
	$$
	P_{a}f(x):=\frac{1}{\sqrt{4\pi a}}\int_{\mathbb{R}}f(y)e^{-\frac{(x-y)^2}{4a}}dy.
	$$
	A direct computation leads  to
	\begin{align*}
	\partial_xu(r,x)&=\partial_x(P_{r-r_0 }g(x))+\int_{r_0 }^r\partial_x\bigl(P_{r-w}(\kappa_1(w,\cdot)\partial_xu(w,\cdot)+\kappa_2(w,\cdot))(x)\bigr)dw.
	\end{align*}
	Note that
	\begin{align*}
	\partial_x(P_{r-r_0 }g(x))&=\frac{-1}{2(r-r_0 )\sqrt{4\pi (r-r_0 )}}\int_{\mathbb{R}}(x-y)g(y)e^{-\frac{(x-y)^2}{4(r-r_0 )}}dy,
	\end{align*}
	from which
	\begin{align*}
	|\partial_x(P_{r-r_0 }g(x))|&\leq \frac{C_0}{2(r-r_0 )\sqrt{4\pi (r-r_0 )}}\int_{\mathbb{R}}|x-y|e^{-\frac{(x-y)^2}{4(r-r_0 )}}dy=\frac{C_0}{\sqrt{\pi(r-r_0 ) }}.
	\end{align*}
	A similar computation also yields  that
	\begin{align*}
	\int_{r_0 }^r\partial_x\bigl(P_{r-w}(\kappa_1(w,\cdot)\partial_xu(w,\cdot)+\kappa_2(w,\cdot))(x)\bigr)dw&\leq \int_{r_0 }^r\frac{C_1\|\partial_xu(w,\cdot)\|_\infty+C_2}{\sqrt{\pi(r-w)}} dw.
	\end{align*}
	As a result,
	\begin{align*}
	\|\partial_xu(r,\cdot)\|_\infty&\leq \phi_0(r)+C_1\int_{r_0 }^r\frac{\|\partial_xu(w,\cdot)\|_\infty}{\sqrt{\pi(r-w)}} dw,
	\end{align*}
	where $$
	\phi_0(r):=\frac{C_0}{\sqrt{\pi(r-r_0 )}}+\frac{2C_2\sqrt{r-r_0 }}{\sqrt{\pi}}.
	$$
	From the Gronwall inequality,
	\begin{align*}
	\|\partial_xu(r,\cdot)\|_\infty&\leq \phi_0(r)+\int_{r_0 }^r\frac{C_1\phi_0(w)}{\sqrt{\pi(r-w)}}\exp\Bigl(\int_{w}^r\frac{dl}{\sqrt{\pi(r-l)}}\Bigr)dw\\
	&=\phi_0(r)+\int_{r_0 }^r\frac{C_1\phi_0(w)}{\sqrt{\pi(r-w)}}\exp\Bigl(2\sqrt{\frac{r-w}{\pi}}\Bigr)dw\\
	&\leq \phi(r),
	\end{align*}
	where
	$$
	\phi(r):=\phi_0(r)+\frac{C_1}{\pi}\exp\Bigl(2\sqrt{\frac{r_2-r_0 }{\pi}}\Bigr)\bigl(\pi C_0+4C_2\sqrt{r_2-r_0 }\bigr),
	$$
	and the last inequality was obtained by using $\int_{r_0 }^r1/\sqrt{(r-w)(w-r_0 )}dw =\pi.$ This finishes our proof.
\end{proof}

With the help of Lemma \ref{pde:lem}, we obtain some controls on the spacial derivatives of $\Phi_\gamma.$

\begin{lemma}\label{pde:lem3}
	Let $s_1\in(0,1)$. For any $\gamma\in \mathcal{U}_d,$
	\begin{align}
	\label{pde:prop2:eq0}
	\sup_{(s,x)\in[0,1)\times\mathbb{R}}|\partial_x\Phi_\gamma(s,x)|\leq 1
	\end{align}
	and for $k\geq 2$
	\begin{align}\label{pde:prop2:eq1}
	\sup_{(s,x)\in[0,s_1]\times\mathbb{R}}|\partial_x^k\Phi_\gamma(s,x)|\leq F_k(\gamma(s_1)),
	\end{align}
	where $F_k$ is nonnegative continuous on $[0,\infty)$ depending on $s_1$ only and independent of $\gamma.$
\end{lemma}

\begin{proof}
	Assume that $\gamma\in \mathcal{U}_d.$ One can explicitly solve $\Phi_\gamma$ in the classical sense by performing the Cole-Hopf transformation. In fact, if $\gamma=\sum_{i=0}^{m}a_i1_{[q_{i},q_{i+1})}$ for some sequences
	\begin{align*}
	&0=q_0<q_1<\cdots<q_m<q_{m+1}=1,\\
	&0\leq a_0\leq a_1\leq \cdots\leq a_{m-1}\leq a_{m}<\infty,
	\end{align*}
	then for any $1\leq i\leq m,$
	\begin{align*}
	\Phi_{\gamma}(s,x)&=\frac{1}{a_i}\log \e\exp a_i\Phi_\gamma(q_{i+1},x+z\sqrt{\xi'(q_{i+1})-\xi'(s)})
	\end{align*}
	for all $(s,x)\in [q_i,q_{i+1})\times\mathbb{R},$ where $z$ is a standard normal random variable. Using the initial condition $\Phi_\gamma(1,x)=|x|$ and an iteration argument, it can be easily checked that for any $k\geq 0$, $\partial_x^k\Phi\in C([0,1)\times\mathbb{R})$ and for any $k\geq 0$ and $s_0\in(0,1)$, $\partial_s\partial_x^k\Phi(s,\cdot) \in C(\mathbb{R})$. Furthermore, it can be verified that \eqref{pde:prop2:eq0} holds, and for any $s_0\in(0,1)$, there exists a constant $C$ such that $$
	\sup_{x\in\mathbb{R}}|\partial_{x}^k\Phi_\gamma(s,x)-\partial_x^k\Phi_\gamma(s',x)|\leq C|s-s'|,\,\,\forall s,s'\in[0,1),\,x\in\mathbb{R},\,\,k\geq 2.$$
	This inequality implies that $s\mapsto \|\partial_x^k\Phi_\gamma(s,\cdot)\|_\infty$ is a continuous function for $s\in[0,1).$ Now define $\zeta(s)=(\xi'(1)-\xi'(s))/2$. Fix $0<s_1<s_0<1$. Set $r_0=\zeta(s_0)$, $r_1=\zeta(s_1)$, and $r_2=\zeta(0)$. Evidently, the function $u(r,x):=\partial_{x}\Phi_\gamma(\zeta^{-1}(r),x)$ satisfies
	\begin{align*}
	\partial_ru(r,x)&=\partial_{xx}u(r,x)+\kappa_1(r,x)\partial_xu(r,x)+\kappa_2(r,x)
	\end{align*}
	for $(r,x)\in [r_0,r_2]\times\mathbb{R}$ with initial condition $\partial_xu(r_0,x)=\partial_x\Psi_\gamma(\zeta^{-1}(r_0),x),$ where
	\begin{align*}
	\kappa_1(r,x)&:=\gamma(\zeta^{-1}(r))\partial_xu(r,x),\\
	\kappa_2(r,x)&:=0.
	\end{align*}
	Since $\|\kappa_1\|_\infty\leq \gamma(\zeta^{-1}(r_0))<\infty$, we can apply Lemma \ref{pde:lem} to get that
	\begin{align*}
	\sup_{(s,x)\in[0,s_1]\times\mathbb{R}}|\partial_x\Phi_\gamma(s,x)|&=\sup_{(r,x)\in[r_1,r_2]\times\mathbb{R}}|\partial_xu(r,x)|\\
	&\leq F(1,\gamma(\zeta^{-1}(r_1)),0)\\
	&= F(1,\gamma(s_1),0),
	\end{align*}
	where $F$ is a nonnegative continuous function on $[0,\infty)^3$ depending only on $s_1.$ Letting $F_1(y)=F(1,y,0)$ gives \eqref{pde:prop2:eq1} with $k=1.$ For $k\geq 2,$ note that
	\begin{align*}
	 \partial_t\bigl(\partial_x^k\Phi_\gamma(s,x)\bigr)&=\frac{\xi''(s)}{2}\Bigl(\partial_{xx}\bigl(\partial_x^k\Phi_\gamma(s,x)\bigr)+2\gamma(s)\partial_x\Phi_\gamma(s,x)\bigl(\partial_x^{k+1}\Phi_\gamma(s,x)\bigr)+K(s,x)\Bigr),
	\end{align*}
	where $K(s,x)$ is the sum of products of spatial derivatives of $\Phi_\gamma$ of order at most $k.$ One may use a similar argument as the case $k=1$ together with an induction procedure to obtain the proof of \eqref{pde:prop2:eq1} for all $k\geq 2$.
\end{proof}

Following a similar definition of the weak solution for the original Parisi PDE in Jagannath-Tobasco \cite{JT2}, we define the weak solution of the Parisi PDE \eqref{pde} as follows.

\begin{definition}[Weak solution]
Let $\gamma \in \mathcal{U}.$ Let $\Phi$ be a continuous function on $[0,1]\times \mathbb{R}$ with essentially bounded weak derivative $\partial_x\Phi.$ We say that $\Phi$ is a weak solution to
	\begin{align*}
	\partial_s\Phi (s,x)&=-\frac{\xi''(s)}{2}\bigl(\partial_{xx}\Phi (s,x)+\gamma(s)\bigl(\partial_x\Phi (s,x)\bigr)^2\bigr)
	\end{align*}
	on $[0,1)\times\mathbb{R}$ with $\Phi(1,x)=|x|$ if
	\begin{align*}
	\int_0^1\int_{\mathbb{R}}\Bigl(-\Phi\partial_s \phi+\frac{\xi''(s)}{2}\Bigl(\Phi\partial_{xx}\phi+\gamma(s)\bigl(\partial_{x}\Phi\bigr)^2\phi\Bigr)\Bigr)dxds+\int_{\mathbb{R}}\phi(1,x)|x|dx=0
	\end{align*}
	for all smooth $\phi$ on $(0,1]\times\mathbb{R}$ with compact support.
	
\end{definition}

\begin{proof}[\bf Proof of Proposition \ref{pde:prop1}]
	Let us pick any $(\gamma_n)_n\subset\mathcal{U}_d$ with weak limit $\gamma.$ From the estimates in Lemma \ref{pde:lem3}, one readily sees that for any $k\geq 0,$ the sequence $(\partial_x^k\Phi_{\gamma_n})_{n\geq 1}$ is equicontinuous and uniformly bounded on $[0,s_1)\times\mathbb{R}$ for all $s_1\in(0,1)$. Thus, from the Arzela-Ascoli theorem combined with a diagonal process, on any $[0,s_1]\times[-M,M],$ $(\partial_x^k\Phi_{\gamma_n})_{n\geq 1}$ converges uniformly on any compact subset of $[0,1)\times\mathbb{R}$ for any $k\geq 0$. Here, without loss of generality, we use the same sequence $(\gamma_n)_{n\geq 1}$ instead of adapting a subsequence for notational convenience. The above discussion makes the following well-defined,
	$$
	\Phi_{\gamma}(s,x):=\lim_{n\rightarrow\infty}\Phi_{\gamma_n}(s,x)
	$$
	for $(s,x)\in[0,1]\times\mathbb{R}.$ Note that for $s\in[0,1)$, $\Phi_{\gamma}(s,x)$ is differentiable in $x\in\mathbb{R}$ and satisfies $\|\partial_x\Phi_\gamma(s,\cdot)\|_\infty<\infty$. Since $\Phi_{\gamma_n}$ satisfies
	\begin{align*}
	\partial_s\Phi_{\gamma_n}(s,x)&=-\frac{\xi''(s)}{2}\bigl(\partial_{xx}\Phi_{\gamma_n}(s,x)+\gamma_n(s)\bigl(\partial_x\Phi_{\gamma_n}(s,x)\bigr)^2\bigr)
	\end{align*}
	for $(s,x)\in[0,1)\times\mathbb{R}$ with boundary condition $\Phi_{\gamma_n}(1,x)=|x|,$ passing to the limit gives
	\begin{align*}
	\int_0^1\int_{\mathbb{R}}\Bigl(-\Phi_{\gamma}\partial_s \phi+\frac{\xi''(s)}{2}\Bigl(\Phi_{\gamma}\partial_{xx}\phi+\gamma(s)\bigl(\partial_{x}\Phi_\gamma\bigr)^2\phi\Bigr)\Bigr)dxds+\int_{\mathbb{R}}\phi(1,x)|x|dx=0
	\end{align*}
	for all smooth functions $\phi$ on $(0,1]\times\mathbb{R}$ with compact support. Therefore, $\Phi_\gamma$ exists in the weak sense and the fact that it satisfies the Lipschitz property \eqref{lip} follows by applying \eqref{MaxFormula} and noting that $|x|$ has Lipschitz constant $1.$ Note that this Lipschitz property also implies that the definition of $\Phi_\gamma$ is independent of the choice of the sequence $(\gamma_n)_{n\geq 1}$. To see the uniqueness of $\Phi_\gamma$, it can be obtained via a fixed point argument identical to \cite[Lemma 13]{JT2}. These together give $(i)$. The above discussion and Lemma \ref{pde:lem3} imply $(ii)$. Since one can also pick this sequence $(\gamma_n)$ from $\mathcal{U}_c$ and perform a similar procedure as above, we also obtain $(iii)$.
\end{proof}


\bibliography{EnergyLandscape_arxiv}
\bibliographystyle{plain}

\end{document}